\documentclass[final,onefignum,onetabnum]{siamart220329}
\usepackage{bm}
\usepackage{amsfonts}
\usepackage{amssymb}
\usepackage{amsmath}
\usepackage{graphicx}
\usepackage{epstopdf}
\usepackage{algorithmic}
\usepackage{enumitem}
\usepackage{booktabs}
\usepackage{multirow}
\usepackage{amsopn}
\usepackage{changes} 
\ifpdf
\DeclareGraphicsExtensions{.eps,.pdf,.png,.jpg}
\else
\DeclareGraphicsExtensions{.eps}
\fi

\DeclareMathOperator{\spans}{span}
\DeclareMathOperator{\diag}{diag}
 
\newcommand{\cov}{\mathrm{cov}}
\newcommand{\EE}{\mathrm{E}}
\newcommand{\NN}{\mathcal{N}}

\newsiamremark{remark}{Remark}
\newsiamthm{exper}{Experiment}
 
\allowdisplaybreaks[4]
 
\headers{AN LSQR-BASED NULL SPACE ALGORITHM}{JINZHI HUANG}

\title{An LSQR-based algorithm for large-scale null space computations\thanks{Submitted to the editors DATE.
		\funding{This work was supported by the Youth Fund of the National Science Foundation of China (Grant No.~12301485) and the Jiangsu Province Youth Science and Technology Talent Support Program (Grant No.~JSTJ-2025-828).}}}

\author{Jinzhi Huang\thanks{School of Mathematical Sciences, Soochow 
		University, 215006 Suzhou, China 
		(\url{jzhuang21@suda.edu.cn}).}}
 
\begin{document}

\maketitle
 
\begin{abstract}
Computing the null space and null vectors of large-scale matrices is 
a fundamental task in numerical linear algebra and scientific computing. 
In this paper, an LSQR-based algorithm, termed LSQRNV, is proposed to 
compute a null vector of a large-scale rank-deficient matrix $A$ 
from an initial vector. 
The theoretical convergence properties of the algorithm are analyzed, 
demonstrating that it converges to a numerical null vector of $A$ at 
a rate dictated by its numerical condition number, and a rigorous 
accuracy bound is derived for the resulting approximation.
By integrating a deflation technique with a tailored termination 
criterion, LSQRNV is extended to LSQRNS, which computes an orthonormal 
basis for the numerical null space of $A$ and explicitly determines its nullity. 
The aforementioned accuracy bound is rigorously generalized to 
the computed approximate numerical null space. 
Furthermore, with appropriate parameter settings, 
LSQRNV efficiently determines whether a large matrix is numerically 
rank-deficient or has full column rank. 
Numerical experiments corroborate the theoretical results, demonstrating 
the robustness, efficiency, and effectiveness of LSQRNS for large-scale 
null-space computations.   	  
\end{abstract}
 
\begin{keywords}
 null vector, null space, nullity, rank, LSQR algorithm, convergence, deflation
\end{keywords}

\begin{MSCcodes}
 65F10, 15A03, 15A06
\end{MSCcodes}

\section{Introduction}\label{sec:1} 
Null space and null vectors are intrinsic properties of a matrix, 
and they play important roles in a variety of applications, such as 
determining the rank and range space of a matrix 
\cite{golub2012matrix,saad2003}, conducting linear discriminant 
analysis \cite{guo2006null,zhang2016learning}, developing 
torque-controlled and autonomous robots 
\cite{dietrich2015overview,lin2015learning}, mitigating 
catastrophic forgetting in continual learning 
\cite{wang2026vpt,zhao2023rethinking},  
and solving linear inverse problems in image restoration 
\cite{gualdron2026gsnr,wang2023zero}. 
 
For a general matrix $A\in\mathbb{R}^{M\times N}$, Coleman and 
Pothen \cite{coleman1986null,coleman1987null} propose a pair of  
two-stage algorithms that utilize bipartite matching to compute 
the fundamental and triangular bases for its null space 
$\mathcal{N}(A)$. Gilbert and Heath 
\cite{gilbert1987computing} develop several QR factorization-based 
algorithms to construct sparse bases for $\mathcal{N}(A)$. 
However, the exorbitant computational and storage demands of these 
algorithms render them prohibitive for large-scale matrices. 
In response, various solvers tailored to specific matrix structures 
have emerged. For instance, Berry et al. \cite{berry1985algorithm} 
propose an algorithm based on Gaussian elimination and orthogonal 
factorization for banded matrices. 
Gotsman and Toledo \cite{gotsman2008computation} employ sparse LU 
factorization with partial pivoting for rectangular matrices with 
small nullity. Foster and Davis \cite{foster2013algorithm} employ 
rank-revealing sparse QR factorization for matrices with very small 
or large nullity, whereas Park and Nakatsukasa \cite{park2023fast} 
design sketch-and-solve methods for tall and skinny matrices. 
Crucially, all these methods are based on certain factorizations of 
$A$ or relevant matrices, restricting their practical applicability 
primarily to matrices with particular structures. 
For a general large-scale  sparse matrix $A$, Kressner and Shao 
\cite{kressner2026randomized} propose a randomized small-block 
Lanczos method that computes an orthonormal basis for 
$\mathcal{N}(A)$ by applying the block Lanczos algorithm 
\cite{golub1981block} to $A^TA$.  
However, working with the cross-product matrix $A^TA$ inherently 
squares the singular values of $A$. In the presence of moderately 
small singular values of $A$, this squaring effect clusters the 
eigenvalues of $A^TA$ densely near the origin, which not only  
degrades the convergence rate but also obscures the numerical 
separation between the true null space and the noise subspace, 
thereby compromising the accuracy of the computed basis. 
Moreover, Lanczos process-based methods incur substantial 
computational overhead from partial reorthogonalization and 
prohibitive storage costs to maintain all Lanczos vectors, 
while implicit restart compromises the overall convergence.
 
In this paper, we propose a novel algorithm for constructing 
\textit{an orthonormal basis} for the null space 
$\mathcal{N}(A)$ of a general, large-scale, rank-deficient matrix $A$.
Our method offers three distinct advantages over prior work. 
First, it avoids explicit factorizations of $A$ or associated 
matrices, thereby bypassing the prohibitive fill-in typically 
associated with direct methods. Second, it operates directly 
on $A$ rather than the cross-product matrix $A^TA$, preserving 
both rapid convergence and high numerical precision. 
Third, it eliminates the need for partial reorthogonalization, 
implicit restarting, and the storage of past Lanczos vectors, 
yielding significant computational efficiency and a minimal 
memory footprint. 
  
Our work originates from the observation that the range space 
$\mathcal{R}(A^T)$ and the null space $\mathcal{N}(A)$ are 
orthogonal complements in $\mathbb{R}^N$. 
Any vector $y\in\mathbb{R}^N$ admits a unique orthogonal decomposition:
\begin{equation}\label{ydecompose}
	y = A^Tz+x \quad \mbox{with} \quad 
	z\in\mathcal{R}(A), \quad x\in\mathcal{N}(A). 
\end{equation}
Here, $x$ is the orthogonal projection of $y$ onto 
$\mathcal{N}(A)$, and the uniqueness of $z\in\mathcal{R}(A)$ 
is guaranteed by $\mathcal{R}(A)\perp\mathcal{N}(A^T)$.
Consequently, computing the null vector $x=y-A^Tz$ reduces 
to first determining the unknown $z\in\mathcal{R}(A)$.
To this end, notice that 
\begin{equation}\label{leastsq}
	z= \arg\min_{\tilde z\in\mathcal{R}(A)} \|y-A^T\tilde z\|. 
\end{equation}
We can employ the standard Krylov subspace 
method, such as the CGLS algorithm 
\cite[Sec.~7.4]{bjorck1996numerical} 
or its equivalent variant LSQR \cite{paige1982lsqr}, 
to solve above least squares problem and obtain 
$z\in\mathcal{R}(A)$. 
In this paper, we adopt the LSQR algorithm due to its superior 
numerical stability, and adapt it to efficiently compute the 
null vector $x$ of $A$ from a given $y\in\mathbb{R}^N$, 
yielding the LSQR-based Null Vector (LSQRNV) algorithm.
  
We rigorously analyze the convergence of LSQRNV and show that for a 
user-prescribed stopping threshold $\varepsilon>0$, it converges to 
a numerical null vector with an accuracy bound of 
$\mathcal{O}(\varepsilon)$ at a rate governed by 
$(\kappa-1)/(\kappa+1)$, where $\kappa$ is the effective 
numerical condition number. 
Furthermore, by integrating LSQRNV with a subspace deflation 
technique and an effective termination criterion, 
we develop the LSQRNS algorithm, which computes an orthonormal 
basis for the numerical null space of $A$ and explicitly 
determines its typically unknown nullity $n$. 
We also show that, with appropriate parameters, the algorithm 
effectively classifies whether a given matrix is numerically 
rank-deficient or has full column rank. 
This capability is crucial for safeguarding large-scale 
computations, as identifying the rank property is a 
prerequisite for selecting appropriate solvers in linear systems, 
determining the model order in factor-analytical approaches, 
and identifying signal subspaces in processing applications. 
Extensive numerical experiments corroborate the theoretical 
results, demonstrating the robustness, efficiency, 
and effectiveness of LSQRNS.
 
The rest of this paper is organized as follows. 
Section~\ref{sec:2} introduces the LSQRNV algorithm. 
Section~\ref{subsec:2-1} provides a rigorous analysis 
of the  convergence and accuracy of LSQRNV.  
Section~\ref{sec:3} develops the LSQRNS algorithm and 
derives accuracy bounds for the computed numerical null space. 
Section~\ref{sec:5} presents the numerical experiments, 
and section~\ref{sec:6} concludes the paper.
 
\section{The LSQRNV algorithm}\label{sec:2} 
For ease of presentation, without loss of generality, 
we assume that the vector $y$ in 
\eqref{ydecompose}--\eqref{leastsq} 
has unit length, i.e., $\|y\|=1$. 
 
The Lanczos bidiagonalization (LBD) process is originally 
proposed by Golub and Kahan \cite{golub1965calculating}, 
and further developed by Paige and Saunders \cite{paige1982lsqr} 
with lower bidiagonal variant. 
Given an initial vector $v_1=y$, the $i$-step lower bidiagonal 
LBD process (if not break down, i.e., $\alpha_j,\beta_j\neq0$, 
$1\leq j\leq i$) can be written as the matrix form: 
\begin{equation}\label{LBDmat}
	\left\{\begin{aligned}
		&A^TU_i=V_{i+1}B_i  \\[0.5em]
		&AV_{i+1}=U_{i} B_{i}^T+\alpha_{i+1}u_{i+1}e_{i+1,i+1}^T  
	\end{aligned}\right.
	\qquad\mbox{with}\qquad
	B_i=\begin{bmatrix}
		\alpha_1\!\!&&\\[-0.6em] \beta_1\!\!&\!\!\ddots&\\[-0.6em]
		 &\!\!\ddots\!\!&\!\!\alpha_{i}\\ &&\!\!\beta_{i}
	\end{bmatrix},
\end{equation} 
where $e_{j,i}$ denotes the $j$th column of the $i\times i$ 
identity matrix $I_i$, and the columns $u_1,\cdots,u_i$ and 
$v_1,\cdots,v_{i+1}$ of $U_i\in\mathbb{R}^{M\times i}$ and 
$V_{i+1}\in\mathbb{R}^{N\times(i+1)}$ form orthonormal bases 
of the $i$- and $(i+1)$-dimensional Krylov subspaces
\begin{equation}\label{searchspace}
\mathcal{U}_i=\mathcal{K}_i(AA^T,u_1)
\qquad\mbox{and}\qquad 
\mathcal{V}_{i+1}=\mathcal{K}_{i+1}(A^TA,v_1)	
\end{equation}
generated by the matrices $AA^T$, $A^TA$ and the vectors 
$u_1=\frac{Av_1}{\alpha_1}$, $v_1$, respectively. 
 
At the $i$th step, the LSQR algorithm for the least squares 
problem in \eqref{leastsq} seeks  an approximate solution 
$z_i\in\mathcal{U}_i$ that minimizes the residual norm, 
that is, the residual 
\begin{equation}\label{xiLSQR} 
	x_i=y-A^Tz_i 
\end{equation}  
satisfies 
\begin{equation}\label{resLSQR0} 
 \|x_i\| = \min_{\tilde z\in\mathcal{U}_i} \|y-A^T\tilde z\|.
\end{equation} 
Write $z_i=U_id_i$ for some $d_i\in\mathbb{R}^{i}$. 
Since $y=v_1$, it is known from 
\eqref{LBDmat}--\eqref{xiLSQR} that  
\begin{equation}\label{resLSQR}
	x_i =  v_1-A^TU_id_i = V_{i+1}(e_{1,i+1}-B_id_i),  
\end{equation}
where, by \eqref{resLSQR0}, $d_i$ solves the following 
$(i+1)\times i$ least squares problem
\begin{equation}\label{defci}
	d_i=\arg \min_{d\in\mathbb{R}^{i}}\|e_{1,i+1}-B_id\|. 
\end{equation} 
Compute the QR factorization of the $(i+1)\times(i+1)$ 
upper Hessenberg matrix
\begin{equation}\label{QRBe}
	\begin{bmatrix}B_i&e_{1,i+1}\end{bmatrix} 
	= \widetilde Q_i\widetilde R_i 
	= \begin{bmatrix}Q_i&q_{i+1}\end{bmatrix} \cdot 
	\begin{bmatrix}R_i&f_i\\\bm{0}&\varphi_{i}\end{bmatrix},
\end{equation} 
where $\widetilde Q_i\in\mathbb{R}^{(i+1)\times(i+1)}$ is 
orthogonal and $R_i\in\mathbb{R}^{i\times i}$ is upper bidiagonal. 
As long as the $i$-step LBD process does not break down, 
$[B_i, e_{1,i+1}]$ is nonsingular, meaning that $R_i$ is 
nonsingular and $\varphi_i\neq 0$. Therefore, applying 
\eqref{QRBe} to \eqref{defci} yields
\begin{equation}\label{defci2} 
	d_i=\arg\min_{d\in\mathbb{R}^{i}}
	\left\|\widetilde Q_i\left(
	\begin{bmatrix}f_i\\ \varphi_{i}\end{bmatrix}
	-\begin{bmatrix}R_i\\ \bm{0} \end{bmatrix}d\right)\right\| 
	=\arg\min_{d\in\mathbb{R}^{i}}\left\|
	\begin{bmatrix}f_i-R_id\\\varphi_{i} \end{bmatrix}\right\|
	=R_i^{-1}f_i.  
\end{equation} 
As a result, inserting \eqref{QRBe} and \eqref{defci2} into 
\eqref{resLSQR}, we obtain  
\begin{equation}\label{resLSQR2}
	x_i = V_{i+1}\begin{bmatrix}Q_i&q_{i+1}\end{bmatrix}
	\begin{bmatrix}f_i-R_id_i\\\varphi_{i}\end{bmatrix}  
	= \varphi_{i} V_{i+1} q_{i+1},  
\end{equation}   
which, by \eqref{xiLSQR}, serves as a reasonable 
approximation to the null vector $x$ of $A$.

Note that for a null vector of $A$, it is its direction 
rather its size that matters. Therefore, for the approximation 
$x_i$ in \eqref{resLSQR2}, supposing that $\varphi_i\neq0$, 
we can simply compute and store its normalized counterpart 
\begin{equation}\label{xi}
	\hat x_i = x_i/\varphi_i =V_{i+1}q_{i+1} 
\end{equation}
during the computations. Specifically, remind that the 
orthogonal matrix $\widetilde Q_i$ in \eqref{QRBe} is 
the product of $i$ Givens rotations designed to successively 
eliminate the subdiagonals $\beta_1,\cdots,\beta_i$ of $B_i$ 
in \eqref{LBDmat}. It is known shown in \cite{paige1982lsqr} that 
\begin{equation*}
	\widetilde Q_i=G_{1}^{(i+1)}G_{2}^{(i+1)}\cdots G_{i}^{(i+1)} 
	\quad\mbox{with}\quad
	G_{j}^{(i+1)} =   
	\begin{bmatrix} 
		I_{j-1}\!\!&&& \\[-0.2em] &{c}_j&{s}_j&  \\
		&-{s}_j&{c}_j& \\[-0.2em] &&&\!\! I_{i-j} 
	\end{bmatrix}, 
\end{equation*} 
where 
\begin{equation}\label{cjsj}
c_j = \frac{c_{j-1}\alpha_j}
{\sqrt{\beta_j^2+c_{j-1}^2\alpha_j^2}}
\qquad\mbox{and}\qquad
s_j = \frac{-\beta_j}{\sqrt{\beta_j^2+c_{j-1}^2\alpha_j^2}}	
\end{equation}
for $j=1,\dots,i$ with $c_0=1$.    
Exploiting the definition of $q_{i+1}$ in \eqref{QRBe}, 
we deduce 
\begin{eqnarray}\label{qj}
	q_{i+1}&=&
	\widetilde Q_ie_{i+1,i+1}=G_{1}^{(i+1)}G_{2}^{(i+1)}
	\cdots G_{i-1}^{(i+1)} \cdot G_{i}^{(i+1)}e_{i+1,i+1} \nonumber \\
	&=&\begin{bmatrix}G_{1}^{(i)}G_{2}^{(i)}
		\cdots G_{i-1}^{(i)} \!&\\&\!1\end{bmatrix}
	\begin{bmatrix}s_ie_{i,i}\\c_i \end{bmatrix} 
	=\begin{bmatrix}\widetilde Q_{i-1}\!&\\&\!1\end{bmatrix}
	\begin{bmatrix}s_ie_{i,i}\\c_i \end{bmatrix}    \nonumber \\
	&=& 
	\begin{bmatrix}s_i\widetilde Q_{i-1}e_{i,i}\\c_i\end{bmatrix}  
	=\begin{bmatrix}s_iq_i\\c_i \end{bmatrix},
\end{eqnarray} 
where $q_i\in\mathbb{R}^i$ is the last column of the Q-factor 
$\widetilde Q_{i-1}$ of $[B_{i-1},e_{1,i}]\in\mathbb{R}^{i\times i}$ 
in the $(i-1)$th step; see \eqref{QRBe}. 
As a consequence, combining \eqref{xi} and \eqref{qj}, we obtain
\begin{equation*}
	\hat x_i 
	= [V_i,v_{i+1}]\begin{bmatrix}s_iq_i \\ c_i \end{bmatrix} 
	= s_iV_iq_i+c_iv_{i+1} 
	= s_i\hat x_{i-1}+c_iv_{i+1}. 
\end{equation*}
This means, the unit-length approximate null vector $\hat x_i$ 
of $A$ can be updated efficiently.
Furthermore, premultiplying $q_{i+1}^T$ both sides 
of \eqref{QRBe} and taking the last elements, and then exploiting 
\eqref{qj} yields
\begin{equation}\label{varphi}
	\varphi_i=q_{i+1}^Te_{1,i+1} 
	=\begin{bmatrix}s_iq_i^T\!\! &\!\! c_i  \end{bmatrix} 
	\begin{bmatrix}	e_{1,i} \\ 0 \end{bmatrix} 
	=s_iq_i^Te_{1,i}=s_i\varphi_{i-1},\qquad
	i\geq1. 
\end{equation} 
Therefore, the residual norm $\|x_i\|=|\varphi_i|$ can also be 
updated successively during the calculations, where 
$\varphi_0=\|x_0\|=\|y\|=1$ as $z_0=\bm{0}$.  

The residual of $z_i$ with respect to the equivalent normal 
equation of the least squares problem in \eqref{leastsq} 
is defined by
\begin{equation}\label{residual}
	r_i=Ay-AA^Tz_i = Ax_i.
\end{equation} 
Clearly, $r_i$ vanishes if and only if $z_i=z$ and $x_i=x$. 
Therefore, for a user-prescribed threshold $\varepsilon>0$, once 
\begin{equation}\label{convergence}
	\|r_i\| \leq  \|A\|\|x_i\| \cdot \varepsilon, 
\end{equation} 
we claim that $z_i$ and $x_i$ have converged and stop the algorithm. 
Substituting \eqref{resLSQR2} into \eqref{residual}, and 
exploiting \eqref{LBDmat}, \eqref{QRBe} and \eqref{qj}, we obtain
\begin{eqnarray}\label{ri}
	\|r_i\|&=&\|\varphi_{i}AV_{i+1} q_{i+1}\|
	=\| \varphi_{i}\left(U_{i}B_{i}^T+
	\alpha_{i+1}u_{i+1}e_{i+1,i+1}^T\right)q_{i+1}\|   \nonumber\\
	&=&\| \varphi_{i}\left(U_{i}R_i^TQ_{i}^T+
	\alpha_{i+1}u_{i+1}e_{i+1,i+1}^T\right)q_{i+1}\| \nonumber\\
	&=&\alpha_{i+1} |\varphi_{i}\cdot e_{i+1,i+1}^Tq_{i+1}|
	=\alpha_{i+1}|\varphi_{i}c_i|,  
\end{eqnarray} 
meaning that $\|r_i\|$ can be calculated efficiently 
during the computations. 
As a result, by \eqref{resLSQR2}, we terminate the 
LSQR algorithm once 
\begin{equation}\label{conLSQR3}
	\alpha_{i+1}|c_i| \leq \tilde\sigma_1\varepsilon,
\end{equation} 
where $\tilde\sigma_1$ is an estimate for $\|A\|$.

\begin{remark}\label{remark2-1} 
If the LBD process breaks down with $\beta_i=0$ for some $i$, 
then \eqref{cjsj} and \eqref{varphi} show that $s_i=0$ and 
$\varphi_i=0$. This yields $x_i=x=\bm{0}$ and $y=A^Tz_i$, revealing 
that the starting vector $y$ in \eqref{ydecompose} has no component 
in $\mathcal{N}(A)$. On the other hand, if a break down occurs with
$\alpha_{i+1}=0$ and $\beta_1,\dots,\beta_{i}\neq0$, it follows 
from \eqref{cjsj} and \eqref{varphi} that $\varphi_i=s_i\dots 
s_0\varphi_0\neq0$. Moreover, \eqref{ri} indicates that $r_i=\bm{0}$, 
meaning that $x_i=x\neq\bm{0}$ is a nontrivial null vector of $A$. 
Therefore, $\alpha_{i+1}=0$ is a lucky event. 
\end{remark}

\begin{algorithm}[htbp]
	\caption{LSQRNV: The LSQR algorithm for the null vector of $A$.}
	\begin{algorithmic}[1]\label{alg2}
		\STATE{Set $\hat x_0= v_1=y$ and $c_0= \varphi_0=1$. 
			Calculate $\alpha_1=\|Av_1\|$ and 
			$u_1=Av_1/\alpha_1$.}
		
		\FOR{$i=1,2,\cdots, N$}
		
		\STATE{Compute $r_i=A^Tu_i-\alpha_{i}v_{i}$ and 
			$\beta_i=\|r_i\|$.}  
		
		\STATE{\textbf{if} $\beta_i=0$ \textbf{then} return $\bm{0}$;
			\textbf{else} compute $v_{i+1}=r_i/\beta_i$. 
			\textbf{fi} \label{step1}}   
		
		\STATE{Compute $t_i=Av_{i+1}-\beta_{i}u_{i}$ and
			$\alpha_{i+1}=\|t_i\|$.} 
		
		\STATE{Calculate $c_i=\frac{c_{i-1}\alpha_i}
			{\sqrt{\beta_i^2+c_{i-1}^2\alpha_i^2}}$ and 
			$s_i=\frac{-\beta_i}
			{\sqrt{\beta_i^2+c_{i-1}^2\alpha_i^2}}$, and 
			update $\varphi_i=s_i\varphi_{i-1}$. \label{step5}}
		
		\STATE{Update $\hat x_i=s_i \hat x_{i-1}+c_iv_{i+1}$. \label{step2}}
		
		\STATE{\textbf{if} $\alpha_{i+1}|c_{i}|\leq \tilde\sigma_1\varepsilon$  
			\textbf{then} return $x_{i}$; \textbf{else} 
			compute $u_{i+1}=t_i/\alpha_{i+1}$. \textbf{fi} \label{step3}}
		
		\ENDFOR
	\end{algorithmic}
\end{algorithm} 

Algorithm~\ref{alg2} summarizes the LSQR-based algorithm for the 
null vector problem \eqref{ydecompose} of a large-scale 
matrix $A$, referred to as LSQRNV hereafter.   
Step~\ref{step3} handles the lucky breakdown 
$\alpha_{i+1}=0$ within the convergence criterion \eqref{conLSQR3}. 
At each iteration, the algorithm requires two matrix-vector 
products, one with $A$ and the other with $A^T$, as well as $5M+8N$ 
flops. During the computations, it requires storing one 
$M$-dimensional vector $u_i$ and two $N$-dimensional vectors, 
$v_i$ and $x_i$.  
 
\section{Convergence and accuracy analysis of LSQRNV}\label{subsec:2-1}

In numerical computations, due to rounding errors, the stored 
matrix $A$ generally is mathematically of full column rank, 
even though its smallest singular values reside at the level 
of $\mathcal{O}(\|A\|\bm{\epsilon_{\mathrm{mach}}})$, where  
$\bm{\epsilon_{\mathrm{mach}}}$ is the machine precision. 
Specifically, for a user-prescribed convergence threshold 
$\varepsilon$ in \eqref{convergence}, assume that the singular 
values $\sigma_1\geq \dots\geq \sigma_{N}$  of $A$ satisfy 
\begin{equation}\label{sig2}
	\frac{1}{\kappa}=:\frac{\sigma_{\ell}} {\sigma_1}
	\gg \varepsilon \gtrsim
	\frac{\sigma_{\ell+1}}{\sigma_1}:= \bm{\epsilon}.
\end{equation} 
Such $A$ is referred to as numerically rank deficient with 
respect to $\bm{\epsilon}$, where $\ell=N-n$ is the numerical 
rank and $\kappa= \sigma_1/\sigma_{\ell}$ is  the numerical 
condition number.  
Consequently, the thin singular value decomposition (SVD) 
of $A$ can be partitioned as
\begin{equation}\label{svdA2}
	A = M\Sigma W^T = \begin{bmatrix}M_{\ell}&M_{\epsilon}\end{bmatrix}
	\begin{bmatrix} \Sigma_{\ell} &\\&\Sigma_{\epsilon}\end{bmatrix}
	\begin{bmatrix}W_{\ell}^T\\W_{\epsilon}^T \end{bmatrix} 
\end{equation}
with $\Sigma_{\ell}=\diag\{\sigma_1,\dots,\sigma_{\ell}\}$ and 
$\Sigma_{\epsilon}=\diag\{\sigma_{\ell+1},\dots,\sigma_{N}\}$. 
We define   
\begin{equation*} 
	\mathcal{R}_{\ell}=\spans\{W_{\ell}\}\qquad\mbox{and}\qquad
	\mathcal{N}_{\epsilon}=\spans\{W_{\epsilon}\} 
\end{equation*} 
as the numerical range space of $A^T$ and the numerical 
null space of $A$ associated with $\bm{\epsilon}$, respectively.  
In this section, we show that the LSQRNV algorithm proposed in 
Section~\ref{sec:2} approaches $\mathcal{N}_{\epsilon}$ at a 
rate bounded by $\frac{\kappa-1}{\kappa+1}$. We also establish 
an accuracy bound for the final computed approximation.
We remark that, up to scaling, LSQRNV generates an identical 
sequence of iterates $x_i$ if $A$ is replaced by $A/\sigma_1$. 
Thus, without loss of generality, we assume $\sigma_1=1$ 
throughout this section to ease the presentation. 
Consequently, all the singular values $\sigma_1,\dots,\sigma_{N}$ 
of $A$ reside in the interval $[0,1]$.
 
Given a normalized initial vector $y\in\mathbb{R}^{N}$, 
we can orthogonally decompose it using the SVD \eqref{svdA2} 
of $A$ as 
\begin{equation}\label{yxz}
	y = y_{\ell} + x_{\epsilon} 
	\quad\mbox{with}\quad
	y_{\ell}=W_{\ell}g_{\ell}\in\mathcal{R}_{\ell},\quad 
	x_{\epsilon}=W_{\epsilon}g_{\epsilon}\in\mathcal{N}_{\epsilon}, 
\end{equation} 
where $g_{\ell}=W_{\ell}^Ty$ and $g_{\epsilon} = W_{\epsilon}^Ty$ 
satisfy 
\begin{equation}\label{g1g2}
    \|g_{\ell}\| = \|y_{\ell}\| 
    = \sin\angle(\mathcal{N}_{\epsilon},y)
    \quad\mbox{and}\quad
    \|g_{\epsilon}\| = \|x_{\epsilon}\| 
    = \cos\angle(\mathcal{N}_{\epsilon},y) . 
\end{equation}
Here, $\angle(\mathcal{N}_{\epsilon},y)$ denotes the 
acute angle between $\mathcal{N}_{\epsilon}$ and $y$.  

Recall from Algorithm~\ref{alg2} that $u_1=Ay/\alpha_1$. 
For the approximation $x_i$ generated at the $i$th step 
of the LSQRNV algorithm, relations 
\eqref{searchspace}--\eqref{resLSQR0} imply that   
\begin{equation}\label{xii}
	x_i = y-A^T\tilde z_i=\psi_i(A^TA)y 
	\qquad\mbox{with}\qquad
	\psi_i=\arg \min_{\psi\in \Psi_i}\|\psi(A^TA)y\|,
\end{equation} 
where $\Psi_i$ denotes the set of polynomials $\psi$ 
of degree at most $i$ satisfying $\psi(0)=1$. 
For subsequent analysis, we choose a specific 
$\widehat \psi_i\in\Psi_i$ such that 
\begin{equation}\label{varphihat}
	\|\widehat\psi_i(\Sigma_{\ell}^2)g_{\ell}\| 
	=\min\limits_{\psi\in \Psi_i}\|\psi(\Sigma_{\ell}^2)g_{\ell}\|
	\leq\min_{\psi\in \Psi_i}\max_{1\leq j 
		\leq \ell}|\psi(\sigma_j^2)| \|g_{\ell}\|
	\leq 2\left(\frac{\kappa-1}{\kappa+1}\right)^{i}\|y_{\ell}\|.
\end{equation}
Here, the inequality follows from 
\cite[pp.~51--52]{greenbaum1997iterative} and \eqref{g1g2}, and 
$\kappa=\frac{\sigma_1}{\sigma_{\ell}}$ is the numerical condition 
number of $A$ with respect to $\bm{\epsilon}$; see \eqref{sig2}. 
Furthermore, since $\psi(0)=1$ for any $\psi\in\Psi_i$, there 
exists a unique polynomial $\phi\in\Phi_{i-1}$, where $\Phi_{i-1}$ 
denotes the space of polynomials of degree at most $i-1$, such that  
\begin{equation}\label{rho}
	\psi(\varsigma)=1+\varsigma\cdot\phi(\varsigma)
	\qquad\mbox{for}\qquad
	\varsigma\in\mathbb{R}.
\end{equation}
In particular, let $\phi_i$ and $\hat\phi_i$ be the unique 
polynomials in $\Phi_{i-1}$ satisfying, 
for any $\varsigma\in\mathbb{R}$, 
\begin{equation}\label{varphistar}
	\psi_i(\varsigma)=1+\varsigma\cdot \phi_{i}(\varsigma)
	\qquad\mbox{and}\qquad
	\widehat\psi_i(\varsigma)=1+\varsigma\cdot \hat\phi_{i}(\varsigma).  
\end{equation} 
To facilitate further discussion, we define the sets
\begin{equation}\label{PhiPi}
    \Psi_i^{\prime} = \{\psi_i,\widehat \psi_i\}
\qquad\mbox{and}\qquad
\Phi_{i-1}^{\prime} =\{\phi_{i},\hat\phi_{i}\}. 
\end{equation}

We first introduce the following lemma for later use.  

\begin{lemma}\label{lemma3}  
	For both polynomials $\phi\in\Phi_{i-1}^{\prime} $ 
	defined in \eqref{PhiPi}, it holds that
	\begin{equation}\label{rhosigma}
	\|\Sigma_{\epsilon}^2 \phi(\Sigma_{\epsilon}^2)\|
	\leq \Delta_i\bm{\epsilon}^2,
\end{equation}
where the matrix $\Sigma_{\epsilon}$ is as in \eqref{svdA2}, 
and the scalar  
\begin{equation}\label{deltai}
    \Delta_i = \max\{\max_{0\leq\varsigma\leq\bm{\epsilon}^2} 
    |\phi_i(\varsigma)|,\max_{0\leq\varsigma\leq\bm{\epsilon}^2}
    |\hat\phi_i(\varsigma)|  \}.
\end{equation} 
\end{lemma} 
 
\begin{proof}
	Recall from \eqref{svdA2} that $\Sigma_{\epsilon}
	=\diag\{\frac{\sigma_{\ell+1}}{\sigma_1},\dots,
	\frac{\sigma_{N}}{\sigma_1}\}$.  
    Since $\frac{\sigma_j}{\sigma_1} \leq 
    \frac{\sigma_{\ell+1}}{\sigma_1} = \bm{\epsilon}$ 
    for $\ell+1\leq j\leq N$, we deduce 
    $\|\Sigma_{\epsilon}\|\leq \bm{\epsilon}$.  
Moreover, for either polynomial $\phi\in\Phi_{i-1}^{\prime}$, 
it follows from \eqref{deltai} that
	\begin{equation*} 
	\|\phi(\Sigma_{\epsilon}^2)\|  
    = \max_{\ell+1\leq j\leq N}|\phi(\sigma_{j}^2)| 	
    \leq \max_{0\leq\varsigma
    	\leq\bm{\epsilon}^2}|\phi(\varsigma)| \leq \Delta_i.  
	\end{equation*} 
Combining these two relations establishes \eqref{rhosigma}. 
\end{proof}

With the aid of Lemma~\ref{lemma3}, we present 
the first main result of this paper. 
 
\begin{theorem}\label{thm1}
Suppose that the initial vector $y\not\in\mathcal{R}_{\ell}$, 
and the scalar $\Delta_i$ defined in \eqref{deltai} fulfills
    $\Delta_i\bm{\epsilon}^2<1$. 
Then the $i$th iterate $x_i$ produced by LSQRNV satisfies 
	\begin{equation}\label{conv}
		\sin\angle(\mathcal{N}_{\epsilon},x_i) 
		\leq \frac{2\tan\angle(\mathcal{N}_{\epsilon},y)}
		{1-\Delta_i\bm{\epsilon}^2}
		\left(\frac{\kappa-1}{\kappa+1}\right)^{i}
		+ \frac{2\sqrt{\Delta_i}\bm{\epsilon}}
		{1-\Delta_i\bm{\epsilon}^2} ,  
	\end{equation}  
where $\kappa$ is the numerical condition number of $A$ 
with respect to $\bm{\epsilon}$; see \eqref{sig2}. 
\end{theorem}

\begin{proof}  
For any polynomial $\psi\in\Psi_i$, invoking relations 
\eqref{svdA2},  \eqref{yxz} and \eqref{rho} yields
\begin{equation}\label{varphiATAy}
	\psi(A^TA)y =  W_{\ell}\psi(\Sigma^2_{\ell})g_{\ell} 
	+ W_{\epsilon}\psi(\Sigma^2_{\epsilon})g_{\epsilon}
	= W_{\ell}\psi(\Sigma^2_{\ell})g_{\ell} + x_{\epsilon} 
	+ W_{\epsilon}\Sigma^2_{\epsilon}\phi(\Sigma^2_{\epsilon})g_{\epsilon}, 
\end{equation} 
where the last two terms on the right-hand side are orthogonal 
to the first one. Since the matrix $[W_{\ell},W_{\epsilon}]$ 
is orthogonal, it follows that $\|\psi(A^TA)y\|^2 
= \|\psi(\Sigma^2_{\ell})g_{\ell}\|^2 + \|x_{\epsilon} 
+ W_{\epsilon}\Sigma^2_{\epsilon}\phi(\Sigma^2_{\epsilon})g_{\epsilon}\|^2$. 
In particular, for both polynomials $\psi\in\Psi_i^{\prime}$ 
and their associated $\phi\in\Phi_{i-1}^{\prime}$ defined in 
\eqref{varphistar}, Lemma~\ref{lemma3} together with \eqref{g1g2} 
indicates that 
\begin{equation*}
     \|W_{\epsilon}\Sigma^2_{\epsilon}\phi(\Sigma^2_{\epsilon})g_{\epsilon}\| 
     \leq \|\Sigma^2_{\epsilon}\phi(\Sigma^2_{\epsilon})\|\|g_{\epsilon}\|
\leq\Delta_i\bm{\epsilon}^2\|x_{\epsilon}\|.
\end{equation*} 
Consequently, we obtain 
\begin{align}
	\|\psi(A^TA)y\|^2&\leq \|\psi(\Sigma^2_{\ell})g_{\ell}\|^2 
	+ \left(1+ \Delta_i\bm{\epsilon}^2 \right)^2\|x_{\epsilon}\|^2, \label{upperbound2}\\
	\|\psi(A^TA)y\|^2&\geq \|\psi(\Sigma^2_{\ell})g_{\ell}\|^2
	+ \left(1-\Delta_i\bm{\epsilon}^2 \right)^2\|x_{\epsilon}\|^2. \label{lowerbound2} 
\end{align} 
Evaluating \eqref{upperbound2} and \eqref{lowerbound2} at 
$\psi=\widehat \psi_i$ and $\psi_i$, respectively, and 
utilizing their definitions in  \eqref{xii} and \eqref{varphihat}, 
we arrive at 
\begin{align}
	\|\psi_i(A^TA)y\|^2 &\leq \|\widehat\psi_i(\Sigma^2_{\ell})g_{\ell}\|^2 
	+\left(1+ \Delta_i\bm{\epsilon}^2\right)^2\|x_{\epsilon}\|^2, \label{upperbound3}\\
	\|\psi_i(A^TA)y\|^2 &\geq	\|\widehat\psi_i(\Sigma^2_{\ell})g_{\ell}\|^2 
	+\left(1- \Delta_i\bm{\epsilon}^2\right)^2\|x_{\epsilon}\|^2. \label{lowerbound3}
\end{align}
Coupled with \eqref{xii}, the second relation immediately implies 
\begin{equation}\label{orthdecompose}
	\|x_i\| \geq (1-\Delta_i \bm{\epsilon}^2)\|x_{\epsilon}\|. 
\end{equation} 
Since $y\not\in\mathcal{R}_{\ell}$, \eqref{yxz} ensures that 
$x_{\epsilon}\neq\bm{0}$. 
This fact, combined with the assumption $\Delta_i\bm{\epsilon}^2<1$ 
guarantees that the right-hand side of the above relation is 
strictly positive. 
Furthermore, setting $\psi=\psi_i$ in \eqref{upperbound2} and 
\eqref{lowerbound2}, and incorporating \eqref{lowerbound3} and 
\eqref{upperbound3}, respectively,  we deduce  
\begin{equation}\label{boundkey}
	\|\widehat\psi_i(\Sigma^2_{\ell})g_{\ell}\|^2 
	- 4 \Delta_i\bm{\epsilon}^2 \|x_{\epsilon}\|^2
	\leq \|\psi_i(\Sigma^2_{\ell})g_{\ell}\|^2 
	\leq \|\widehat\psi_i(\Sigma^2_{\ell})g_{\ell}\|^2 
	+ 4 \Delta_i\bm{\epsilon}^2 \|x_{\epsilon}\|^2.
\end{equation} 
Therefore, by the definition of the sine of the angle between 
a subspace and a vector, utilizing \eqref{varphiATAy} and  
\eqref{orthdecompose}--\eqref{boundkey}, we obtain  
\begin{equation}\label{sin} 
	\sin\angle(\mathcal{N}_{\epsilon},x_i) 
	=\frac{\|W_{\ell}^Tx_i\|}{\|x_i\|}=
	\frac{\|\psi_i(\Sigma^2_{\ell})g_{\ell}\|}{\|x_i\|}
	\leq \frac{\|\widehat\psi_i(\Sigma^2_{\ell})g_{\ell}\| 
		+ 2\sqrt{\Delta_i}\bm{\epsilon} \|x_{\epsilon}\|}
	{(1- \Delta_i   \bm{\epsilon}^2)\|x_{\epsilon}\|}. 
\end{equation} 
Finally, substituting \eqref{varphihat} and \eqref{g1g2} 
into this relation establishes \eqref{conv}.      
\end{proof}
 

For the scalar $\Delta_i$ defined in \eqref{deltai}, the 
polynomial structure of $\phi_i$ and $\hat\phi_i$ implies 
that, for sufficiently small $\bm{\epsilon}$,
\begin{equation}\label{deltai2}
	\Delta_i = \max\{|\phi_i(0)|,|\hat\phi_i(0)|\} 
	+ \mathcal{O}(\bm{\epsilon}^2),
\end{equation} 
where, according to \eqref{varphistar}, $\phi_i(0)$ and 
$\hat\phi_i(0)$ represent the slopes of $\psi_i$ and 
$\widehat{\psi}_i$ at the origin, respectively. 
By \eqref{varphihat}, the polynomial
$\widehat \psi_i\in\Psi_i$ is constructed to minimize its 
magnitude at $\sigma_1^2,\dots,\sigma_{\ell}^2$  in a 
weighted least-squares sense.  
For $i\leq\ell$, standard analysis indicates that all
$i$ roots of $\widehat \psi_i$ are located in the interval 
$[\sigma_{\ell}^2,\sigma_1^2]$, and therefore $\widehat\psi_i$ 
decreases monotonically on  $[0,\sigma_{\ell}^2]$.  
Consequently, we have  
$|\hat\phi_i(0)|\sim\frac{1-\widehat\psi_i(\sigma_{\ell}^2)}
{\sigma_{\ell}^2} \leq \kappa^2.$  
In the early phase of LSQRNV, when 
$\|\widehat\psi_i(\Sigma^2_{\ell})g_{\ell}\|\gtrsim 
\mathcal{O}(\bm{\epsilon})$, 
 bound \eqref{boundkey} indicates that $\psi_i\approx\widehat\psi_i$ 
within an error margin of $\mathcal{O}(\bm{\epsilon})$, 
yielding $|\phi_i(0)|\approx |\widehat \phi_i(0)|$ and thus 
\begin{equation}\label{delta}
	\Delta_i\lesssim \kappa^2.
\end{equation}
 
Consequently, Theorem~\ref{thm1} establishes that the iterate 
$x_i$ generated by LSQRNV approaches the numerical null space 
$\mathcal{N}_{\epsilon}$ of $A$ at a rate governed by 
$\frac{\kappa-1}{\kappa+1}$, until an accuracy of 
$\mathcal{O}(\bm{\epsilon})$ is attained. 
Specifically, the better the numerical conditioning of $A$ is, 
the tighter this bound becomes, indicating faster convergence. 
Correspondingly, the richer information $y$ contains regarding 
$\mathcal{N}_{\epsilon}$ (i.e., the smaller 
$\tan\angle(\mathcal{N}_{\epsilon},y)$ is), the sooner this 
accuracy bound is reached.  
Furthermore, exploiting the SVD of $A$ in \eqref{svdA2}, any 
iterate $x_i$ admits the orthogonal decomposition 
\begin{equation*}
	x_i = W_{\ell}f_{\ell}+W_{\epsilon}f_{\epsilon} 
	\quad\mbox{with}\quad
	f_{\ell} =W_{\ell}^T x_i,\quad
	f_{\epsilon}=W_{\epsilon}^{T} x_i.	
\end{equation*}
Then $\sin\angle(\mathcal{N}_{\epsilon},x_i)
=\frac{\|f_{\ell}\|}{\|x_i\|}$ and
$\cos\angle(\mathcal{N}_{\epsilon},x_i)
=\frac{\|f_{\epsilon}\|}{\|x_i\|}$. 
Additionally, the orthogonal decomposition 
$Ax_i=M_{\ell}\Sigma_{\ell}f_{\ell}
+M_{\epsilon}\Sigma_{\epsilon}f_{\epsilon}$ 
yields the bounds $\sigma_{\ell}\|f_{\ell}\|\leq\|Ax_i\|
\leq\sigma_1\|f_{\ell}\|+\sigma_{\ell+1}\|f_{\epsilon}\|$. 
Consequently, utilizing \eqref{residual} and \eqref{sig2}, 
we deduce that 
\begin{equation}\label{accralation}
	\frac{1}{\kappa}\sin\angle(\mathcal{N}_{\epsilon},x_i) 
	\leq \frac{\|r_i\|}{\|A\|\|x_i\|} 
	\leq \sin\angle(\mathcal{N}_{\epsilon},x_i) 
	+ \bm{\epsilon}\cos\angle(\mathcal{N}_{\epsilon},x_i).
\end{equation} 
Given that $\varepsilon\gtrsim\bm{\epsilon}$ from  \eqref{sig2},  
the second inequality above guarantees that LSQRNV terminates 
via the criterion \eqref{convergence} no later than when  
$\sin\angle(\mathcal{N}_{\epsilon},x_i)$ reaches 
$\mathcal{O}(\bm{\epsilon})$. 

As a direct consequence of the first inequality in \eqref{accralation}, 
the following theorem establishes an accuracy bound for the converged 
numerical null vector.  

\begin{theorem}\label{thm2}
	Under the conditions of Theorem~\ref{thm1}, the approximation 
	$\tilde x = x_i$ returned by LSQRNV with the stopping criterion  
	\eqref{convergence} satisfies 
	\begin{equation}\label{accxke}
		\sin\angle(\mathcal{N}_{\epsilon},\tilde x )
		\leq \kappa \varepsilon.
	\end{equation} 
\end{theorem}
 
Given that $\kappa\ll\varepsilon^{-1}$ according to \eqref{sig2}, 
Theorem~\ref{thm2} indicates that the converged $\tilde x$ serves 
as a valid approximation to a specific numerical null vector of $A$ 
in $\mathcal{N}_{\epsilon}$. A better-conditioned matrix $A$ 
generally yields a more accurate $\tilde x$. 
This validates the rationale behind the stopping criterion 
\eqref{convergence}.  
 
On the other hand, once $\sin\angle(\mathcal{N}{\epsilon},x_i)$ 
decreases to $\mathcal{O}(\bm{\epsilon})$, the polynomial 
$\psi_i\in\Psi_i$ in \eqref{xii} begins to minimize its magnitude 
not only at the dominant eigenvalues $\sigma_1^2$, $\dots,\sigma_{\ell}^2$ 
of $A^TA$, but also at the significantly smaller ones, such as 
$\sigma_{\ell+1}^2, \sigma_{\ell+2}^2, \dots, \sigma_N^2$. 
Consequently, both the slope $|\phi_i(0)|$ and, by \eqref{deltai2}, 
the scalar $\Delta_i$ undergo a substantial increase, growing from 
$\mathcal{O}(\kappa^2)$ to 
$\mathcal{O}\left(\frac{\sigma_1^2}{\sigma_{\ell+1}^2}\right) = 
\mathcal{O}\left(\frac{1}{\bm{\epsilon}^2}\right)$, and progressively 
to $\mathcal{O}\left(\frac{\sigma_1^2}{\sigma_{\ell+2}^2}\right), 
\dots,\mathcal{O}\left(\frac{\sigma_1^2}{\sigma_{N}^2}\right)$.  
As a result, bound \eqref{conv} deteriorates, and 
$\sin\angle(\mathcal{N}_{\epsilon},x_i)$ may experience a 
transient increase to $\mathcal{O}(1)$ over subsequent iterations. 
Nevertheless, by introducing the refined tolerance sequence 
$\bm{\epsilon}_k = \frac{\sigma_{\ell+k}}{\sigma_1}$ for 
$k \geq 2$, e.g., $\frac{\sigma_{\ell+2}}{\sigma_1}, 
\frac{\sigma_{\ell+3}}{\sigma_1}, \dots$, we establish the 
subspace inclusion $\mathcal{N}_{\bm{\epsilon}_k} 
\subseteq \mathcal{N}_{\bm{\epsilon}}$, which inherently yields 
$$\sin\angle(\mathcal{N}_{\epsilon},x_i) 
\leq \sin\angle(\mathcal{N}_{\bm{\epsilon}_k},x_i).$$ 
According to Theorem~\ref{thm1},
 $\sin\angle(\mathcal{N}_{\bm{\epsilon}_k},x_i)$ 
will continue to decrease to $\mathcal{O}(\bm{\epsilon}_k)$, 
albeit at a slower convergence rate. 
This implies that after initially reaching $\mathcal{O}(\bm{\epsilon})$,   
$\sin\angle(\mathcal{N}_{\epsilon},x_i)$ may oscillate several times, 
with each oscillation capturing successively smaller singular components, 
ultimately yielding a progressively finer level of accuracy.
 
\begin{figure}[tbhp]
	\centering
	\includegraphics[width=0.85\textwidth]{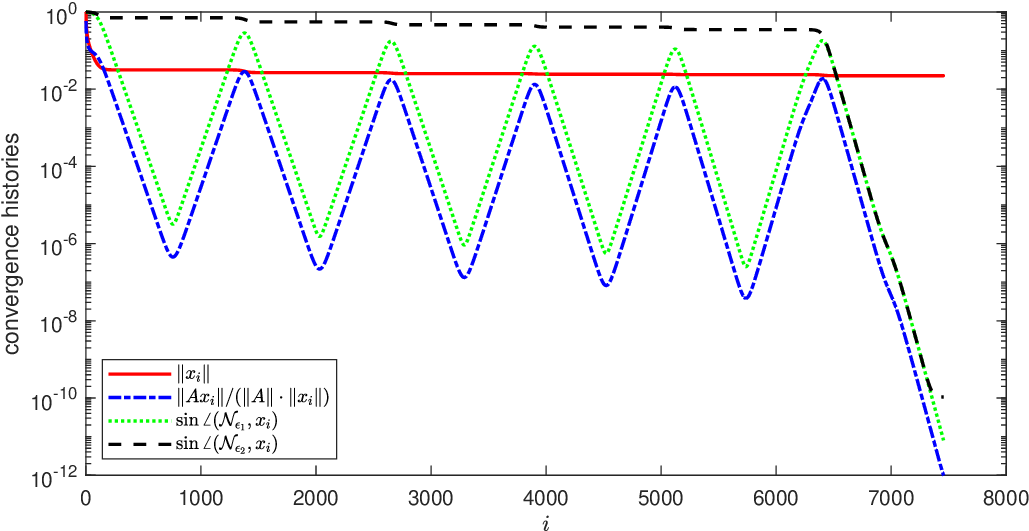}
	\caption{Convergence curve of LSQRNV.}\label{fig0}
\end{figure}
 
\begin{exper}
We applied LSQRNV to a $10000\times 10000$ diagonal matrix $A$ 
using the starting vector $y=\frac{1}{100}[1,\dots,1]^T$. 
The singular values of $A$ are uniformly distributed across 
three distinct clusters: $5$ in $[10^{-14},10^{-13}]$, $5$ in 
$[10^{-8},10^{-7}]$, and the remaining $9990$ in $[10^{-2},1]$. 
We employed two stopping tolerances in \eqref{convergence}, 
$\varepsilon_1=10^{-6}$ and $\varepsilon_2=10^{-12}$, which 
correspond to the thresholds $\bm{\epsilon}_1=10^{-7}$ and 
$\bm{\epsilon}_2=10^{-13}$ in \eqref{sig2}, and the numerical 
condition numbers $\kappa_1=10^2$ and $\kappa_2=10^{8}$, respectively. 
Figure~\ref{fig0} illustrates the iteration histories of the 
norm $\|x_i\|$, the relative residual $\|Ax_i\|/(\|A\|\|x_i\|)$, 
and the accuracy metrics $\sin\angle(\mathcal{N}_{\epsilon_1},x_i)$ 
and $\sin\angle(\mathcal{N}_{\epsilon_2},x_i)$.	    
\end{exper}

As shown in the figure, for $\varepsilon_1 = 10^{-6}$, 
$\sin\angle(\mathcal{N}_{\epsilon_1},x_i)$ initially converged 
rapidly, reaching $9.35\times 10^{-6}$ at iteration $i_1=684$ 
when the stopping criterion \eqref{convergence} was satisfied. 
Conversely, for the stricter tolerance $\varepsilon_2 = 10^{-12}$, 
$\sin\angle(\mathcal{N}_{\epsilon_2},x_i)$ decreased much more 
slowly due to the considerably larger $\kappa_2$, reaching 
$1.06\times 10^{-10}$ at iteration $i_2=7454$ upon meeting 
\eqref{convergence}. 
Additionally, the iterations between $i_1$ and $i_2$ clearly 
illustrate the oscillatory phenomenon analyzed previously. 
These observations confirm the results established in 
Theorems~\ref{thm1} and \ref{thm2}.  
 
\vspace{0.5em}  
   
Theorems~\ref{thm1} and~\ref{thm2} require the fundamental 
assumption that $y\notin \mathcal{R}_{\ell}$. 
We now investigate the algorithm's behavior when $y$ is 
nearly contained in $\mathcal{R}_{\ell}$. 
 
\begin{remark}\label{remark3-4}
Suppose that $\sin\angle(\mathcal{R}_{\ell},y) 
\leq \mathcal{O}(\varepsilon)$. 
By \eqref{yxz}--\eqref{g1g2}, the orthogonal projection 
$x_{\epsilon}$ of $y$ onto $\mathcal{N}_{\epsilon}$ satisfies 
$\|x_{\epsilon}\| = \sin\angle(\mathcal{R}_{\ell},y)
\leq  \mathcal{O}(\varepsilon)$. 
In this case, combining \eqref{xii} with \eqref{upperbound3} 
implies that the $i$th iterate $x_i$ generated by LSQRNV satisfies   
	\begin{equation}\label{xibound} 
		\|x_i\| \leq \|\widehat\psi_i(\Sigma^2_{\ell})g_{\ell}\| 
		+ \left(1+ \Delta_i\bm{\epsilon}^2 \right) \|x_{\epsilon}\|. 
	\end{equation}
	Consequently, we deduce from \eqref{sin} and \eqref{boundkey} that
	\begin{align*} 
		\sin\angle(\mathcal{N}_{\epsilon},x_i) &
		=  \frac{\|\psi_i(\Sigma^2_{\ell})g_{\ell}\|}{\|x_i\|} 
		\geq \frac{\|\widehat\psi_i(\Sigma^2_{\ell})g_{\ell}\| - 
			2 \sqrt{\Delta_i}\bm{\epsilon} \|x_{\epsilon}\|}
		{\|\widehat\psi_i(\Sigma^2_{\ell})g_{\ell}\| + 
			\left(1+ \Delta_i\bm{\epsilon}^2 \right) \|x_{\epsilon}\|} \nonumber \\
		& = \frac{1-2 \sqrt{\Delta_i}\bm{\epsilon} 
			\cdot \frac{\|x_{\epsilon}\|}
			{\|\widehat\psi_i(\Sigma^2_{\ell})g_{\ell}\|}}
		{1 + \left(1+ \Delta_i\bm{\epsilon}^2 \right)\cdot
		\frac{\|x_{\epsilon}\|}
		{\|\widehat\psi_i(\Sigma^2_{\ell})g_{\ell}\|}} 
	\geq \frac{1 - 2 \sqrt{\Delta_i}\bm{\epsilon} }
		{2+\Delta_i\bm{\epsilon}^2 } 
		=\frac{1}{2}+\mathcal{O}(\bm{\epsilon}), 
	\end{align*} 
provided that $\|\widehat\psi_i(\Sigma^2_{\ell})g_{\ell}\|
\geq\|x_{\epsilon}\|$. 
In this scenario, given the stopping tolerance 
$\varepsilon\ll\kappa^{-1}$ from \eqref{sig2}, Theorem~\ref{thm2} 
guarantees that LSQRNV will not meet the convergence 
criterion \eqref{convergence}. 
However, once the monotonically decreasing term 
$\|\widehat\psi_i(\Sigma^2_{\ell}) g_{\ell}\|$ 
drops to $\|x_{\epsilon}\|$, \eqref{xibound} ensures that 
\begin{equation*}   
\|x_i\| 
\leq  \left(2+ \Delta_i\bm{\epsilon}^2 \right) \|x_{\epsilon}\| =
 \mathcal{O}(\varepsilon). 	
\end{equation*} 
In other words, if the starting vector $y$ for LSQRNV lies 
predominantly within $\mathcal{R}_{\ell}$, the monotonically 
decreasing norm $\|x_i\|$ will fall below $\mathcal{O}(\varepsilon)$ 
well before the convergence criterion \eqref{convergence} is satisfied. 
We will leverage this key property to design a termination criterion 
when employing LSQRNV to compute an orthonormal basis for 
$\mathcal{N}_{\epsilon}$; see Section~\ref{subsec:3-2}.
\end{remark} 
 
\section{The LSQRNS algorithm}\label{sec:3} 
This section is devoted to a practical LSQR-based null space 
algorithm for computing an orthonormal basis 
$\{x_1^*,\dots,x_n^*\}$ for the numerical null space 
$\mathcal{N}_{\epsilon}$ of $A$ with the dimension 
$n=\dim(\mathcal{N}_{\epsilon})\geq1$. 

\subsection{Deflation}\label{subsec:3-1}
Suppose we have already obtained $m-1$ reasonably accurate, 
orthonormal approximate numerical null vectors 
$x_{1}^{\mathrm{c}},\dots,x_{m-1}^{\mathrm{c}}$ of $A$, 
where $m\leq n$, and we aim to compute the next one.  
Define   
\begin{equation}\label{defXm}
	X_{m-1}=[x_{1}^{\mathrm{c}},\cdots,x_{m-1}^{\mathrm{c}}]
	\quad\mbox{and}\quad
	\mathcal{X}_{m-1}=\spans\{X_{m-1}\}. 
\end{equation}
It is straightforward to verify that the orthogonal projection 
of any $y\perp\mathcal{X}_{m-1}$ onto $\mathcal{N}_{\epsilon}$ 
is nearly orthogonal to $\mathcal{X}_{m-1}$. 
Therefore, we can apply the LSQRNV algorithm proposed in 
section~\ref{sec:2} with a normalized starting vector 
$y_m\perp\mathcal{X}_{m-1}$ to capture this projection, 
thereby obtaining the next approximate numerical null vector 
$x_m^{\mathrm{c}}\perp\mathcal{X}_{m-1}$. 
This technique is widely known as deflation in the literature 
and is commonly used in numerical linear algebra and 
scientific computing. 

Specifically, at the $m$-th step, we proceed as follows:  
(\romannumeral 1) Choose a normalized vector $y_m\perp X_{m-1}$ 
by setting
\begin{equation}\label{defy}
	y_m=\bar y_m/\|\bar y_m\| 
	\quad\mbox{with}\quad
	\bar y_m = \left(I-X_{m-1}X_{m-1}^{T}\right)\hat y_m,
\end{equation} 
where $\hat y_m \in \mathbb{R}^N$ is a random vector with mean 
$\EE(\hat y_m) = \bm{0}$ and covariance matrix 
$\cov(\hat y_m) = \gamma I$ for some $\gamma>0$.
(\romannumeral 2) Run Algorithm~\ref{alg2} with the initial 
vector $y=y_m$ until convergence to produce a  normalized 
approximate numerical null vector of $A$, denoted by $\tilde x_{m}$. 
(\romannumeral 3) Orthonormalize $\tilde x_{m}$ against $X_{m-1}$ 
to obtain the next 
\begin{equation}\label{xm1}
	x_{m}^{\mathrm{c}} = \frac{(I-X_{m-1}X_{m-1}^T)\tilde x_{m}}
	{\left\|(I-X_{m-1}X_{m-1}^T)\tilde x_{m}\right\|},  
\end{equation}
and expand the basis matrix as  $X_{m}=[X_{m-1},x_{m}^{\mathrm{c}}]$. 
We repeat this procedure for $m=1,\dots,n$ until an $n$-dimensional 
approximate orthonormal basis $X_n$ for the numerical null space 
$\mathcal{N}_{\epsilon}$ of $A$ is established. 
The resulting method is referred to as the LSQR-based null space 
algorithm, abbreviated as LSQRNS.
 
Theorems~\ref{thm1} and \ref{thm2} demonstrate that the success 
of the LSQRNS algorithm relies on choosing suitable starting 
vectors $y_m$ that capture sufficient information about 
$\mathcal{N}_{\epsilon}$. The validity of this selection strategy 
is established by the following theorem.
  
\begin{theorem}\label{thm4}
	Suppose that $x_{1}^{\mathrm{c}},\dots, 
	x_{m-1}^{\mathrm{c}}\in\mathcal{N}_{\epsilon}$.
	Then, the projected vector $\bar y_m$ defined in \eqref{defy},    
	satisfies 
	\begin{equation}\label{exi}
		\qquad\qquad\quad	
		\frac{\EE(\|\bar y_m\|^2\sin^2\angle(\mathcal{N}_{\epsilon},\bar y_m))}
		{\EE(\|\bar y_m\|^2\cos^2\angle(\mathcal{N}_{\epsilon},\bar y_m))}
		=\frac{N-n}{n-m+1},
		\qquad 
		m=1,\cdots,n.
	\end{equation} 
\end{theorem}

\begin{proof}
Let $X_{\perp}\in\mathbb{R}^{N\times(n-m+1)}$ be such that the 
columns of $[X_{m-1},X_{\perp}]$ form  an orthonormal basis 
for $\mathcal{N}_{\epsilon}$. 
By virtue of \eqref{svdA2}, the augmented matrix 
$[W_{\ell}, X_{m-1},X_{\perp}]$ is orthogonal. 
Combined with \eqref{defy}, this implies   
\begin{equation*}
	\bar y_m=W_\ell W_\ell ^T\hat y_m + X_{\perp}X_{\perp}^T\hat y_m,
\end{equation*}
where $X_{\perp}X_{\perp}^T\hat y_m$ is the 
orthogonal projection of $\bar y_m$ onto $\mathcal{N}_{\epsilon}$. 
It then follows that 
	\begin{equation*}
		\|\bar y_m\|\sin\angle\left(\mathcal{N}_{\epsilon},\bar y_m\right) 
		= \|W_\ell^{T}\hat y_m\|
		\quad\mbox{and}\quad
		\|\bar y_m\|\cos\angle\left(\mathcal{N}_{\epsilon},\bar y_m\right)
		= \|X_{\perp}^T\hat y_m\|.
	\end{equation*}
Given that $\mathbb{E}(\hat y_m)=\bm{0}$ 
and $\mathrm{cov}(\hat y_m)=\gamma I$, 
a straightforward calculation reveals 
\begin{align*}
	&\EE\left(\|\bar y_m\|^2 
	\sin^2\angle(\mathcal{N}_{\epsilon},\bar y_m)\right)
	= \gamma\cdot\|W_\ell\|_F^2 = \gamma \cdot (N-n),  \\
	&\EE\left(\|\bar y_m\|^2
	\cos^2\angle(\mathcal{N}_{\epsilon},\bar y_m)\right)
	= \gamma\cdot\|X_{\perp}\|_F^2 = \gamma \cdot (n-m+1), 
\end{align*} 
	where $\|\cdot\|_F$ denotes the Frobenius norm. 
	Dividing these two equations yields \eqref{exi}. 
\end{proof} 
 
Suppose that $y_m$ in \eqref{defy} admits an orthogonal decomposition 
analogous to \eqref{yxz}. By standard continuity arguments, 
its orthogonal projection $x_{m}^{*}$ onto $\mathcal{N}_{\epsilon}$ 
is numerically orthogonal to the previously converged subspace 
$\mathcal{X}_{m-1}$. 
Moreover, since $y_m=\bar y_m/\|\bar y_m\|$,
Theorem~\ref{thm4} demonstrates that in expectation, 
\begin{equation}\label{xm}
	\qquad\qquad   \tan\angle(\mathcal{N}_{\epsilon},y_m)\approx
	\frac{\sqrt{N-n}}{\sqrt{n-m+1}},\qquad m\leq n.
\end{equation} 
Consequently, provided $m\leq n$, the factor  
$\tan\angle(\mathcal{N}_{\epsilon},y_m)$ in Theorem~\ref{thm1} 
remains moderate in size. This guarantees that the LSQRNV 
algorithm will converges within a reasonable total number 
of iterations during the $m$-th outer cycle of the LSQRNS algorithm. 
Additionally, Theorem~\ref{thm2} ensures that the output 
$\tilde x_m$ serves as a valid approximate numerical null 
vector of $A$.

\subsection{Accuracy of the approximate numerical null space}  
This section is devoted to deriving \textit{a priori} accuracy 
bounds for the computed subspaces $\mathcal{X}_{m}$ defined in 
\eqref{defXm} for $m=1,\dots,n$.  
To this end, we first establish the following result regarding  
the accuracy of the $m$-th approximate numerical null vector 
$x_{m}^{\mathrm{c}}$ of $A$ given in \eqref{xm1} for $m=2,\dots,n$. 
 
\begin{lemma}\label{lemma1}
Denote by $\angle(\mathcal{N}_{\epsilon},\mathcal{X}_{m})$ the 
canonical angles between $\mathcal{N}_{\epsilon}$ and 
$\mathcal{X}_{m}$,  $m=1,\dots,n$.   
For $m=2,\dots,n$, the vector $x_{m}^{\mathrm{c}}$ defined in 
\eqref{xm1} satisfies  
\begin{equation}\label{accxm} 
	\sin\angle(\mathcal{N}_{\epsilon},x_{m}^{\mathrm{c}}) 
	\leq \frac{\sin\angle(\mathcal{N}_{\epsilon},\tilde x_{m})
		+ \|\sin\angle(\mathcal{N}_{\epsilon},\mathcal{X}_{m-1})\|
		\cos\angle(\mathcal{X}_{m-1},\tilde x_{m})}
	{\sqrt{1-\cos^2\angle(\mathcal{X}_{m-1},\tilde x_{m})}}. 
\end{equation} 
Furthermore, let $\Delta_{i_m}$ be defined as in \eqref{deltai}, 
where $i_m$ denotes the total number of iterations required for 
LSQRNV to converge during the $m$-th outer cycle of LSQRNS. 
If  $\Delta_{i_m}\bm{\epsilon}^2<1$, then  the cosine term in 
\eqref{accxm}  admits the following bound: 
\begin{equation}\label{cos2}
	\cos\angle(\mathcal{X}_{m-1},\tilde x_{m})
	\leq \|\sin\angle(\mathcal{N}_{\epsilon},\mathcal{X}_{m-1})\|
	\cdot  \frac{\tan\angle(\mathcal{N}_{\epsilon},y_m)}
	{1-\Delta_{i_m}\bm{\epsilon}^2} +\frac{2\sqrt{\Delta_{i_m}}   
		\bm{\epsilon}}{1-\Delta_{i_m}\bm{\epsilon}^2}. 
\end{equation} 
\end{lemma}

\begin{proof}
Recall that the columns of $W_\ell$ and $W_{\epsilon}$ in \eqref{svdA2} 
form orthonormal bases for the mutually orthogonal 
$\mathcal{R}_{\ell}$ and $\mathcal{N}_{\epsilon}$, respectively. 
Utilizing \eqref{xm1} and the definition of the sine and cosine 
of the canonical angles between two subspaces \cite{jiastewart2001}, 
we deduce  
	\begin{align*}
		\sin\angle(\mathcal{N}_{\epsilon},x_{m}^{\mathrm{c}})
		=&\frac{\|W_\ell^Tx_{m}^{\mathrm{c}}\|}{\|x_{m}^{\mathrm{c}}\|} = 
		\frac{\|W_\ell^T(I-X_{m-1}X_{m-1}^T)\tilde x_{m}\|/\|\tilde x_{m}\|}
		{\|(I-X_{m-1}X_{m-1}^T)\tilde x_{m}\|/\|\tilde x_{m}\|} \nonumber\\
		\leq&\frac{\|W_\ell^T\tilde x_{m}\|/\|\tilde x_{m}\|
			+\|W_\ell^TX_{m-1}\|\|X_{m-1}^T\tilde x_{m}\|/\|\tilde x_{m}\|}
		{\sin\angle(\mathcal{X}_{m-1},\tilde x_{m})}      \nonumber\\ 
		=&\frac{\sin\angle(\mathcal{N}_{\epsilon},\tilde x_{m}) 
			+\|\sin\angle(\mathcal{N}_{\epsilon},\mathcal{X}_{m-1})\|
			\cos\angle(\mathcal{X}_{m-1},\tilde x_{m})}
		{\sqrt{1-\cos^2\angle(\mathcal{X}_{m-1},\tilde x_{m})}}, 
	\end{align*}  
thus proving \eqref{accxm}. 
  
Setting $y=y_m$ in \eqref{yxz} and invoking \eqref{xii}, 
\eqref{varphistar}, and \eqref{svdA2}, we deduce
\begin{equation}\label{xim}
	x_{i_m} = (I+A^TA\phi_{i_m}(A^TA))y_m  
	 =y_m+W_{\ell}\tilde g_{\ell}+W_{\epsilon} \tilde g_{\epsilon},   
\end{equation}
where $\tilde g_{\ell}=\Sigma_{\ell}^2\phi_{i_m}(\Sigma_{\ell}^2)g_{\ell}$ 
and $\tilde g_{\epsilon}= \Sigma_{\epsilon}^2\phi_{i_m}
(\Sigma_{\epsilon}^2)g_{\epsilon}$. By Lemma~\ref{lemma3} 
and \eqref{g1g2}, $\tilde g_{\epsilon}$ satisfies  
\begin{equation}\label{tildeg}
	\|\tilde g_{\epsilon}\|\leq\|\Sigma^2_{\epsilon}\phi_{i_m}
	(\Sigma^2_{\epsilon})\|\|g_{\epsilon}\|
	\leq\Delta_{i_m}\bm{\epsilon}^2\|x_{\epsilon}\|.
\end{equation}
By the harmonic extraction property of LSQR applied to \eqref{leastsq}, 
the $i_m$-th residual $x_{i_m}=y_m-A^Tz_{i_m}$ is orthogonal to 
$A^T\mathcal U_{i_m}$ with $\mathcal U_{i_m}$ the current search 
subspace for $z_{i_m}$. Thus, 
$x_{i_m}\perp A^Tz_{i_m}=y_m-x_{i_m} 
=W_{\ell}\tilde g_{\ell}+W_{\epsilon} \tilde g_{\epsilon}$, 
where the latter two terms are mutually orthogonal. 
In other words, alongside \eqref{yxz}, \eqref{xim} implicitly 
provides another  orthogonal decomposition of $y_m$, yielding 
$$
\|x_{i_m}\|^2+\|\tilde g_{\ell}\|^2+\|\tilde g_{\epsilon}\|^2 
= \|y_m\|^2=\|x_{\epsilon}\|^2+\|y_{\ell}\|^2.
$$  
Subtracting $\|x_{i_m}\|^2$ from both sides and applying 
\eqref{orthdecompose} yields
$$
\|\tilde g_{\ell}\|^2\leq \|y_{\ell}\|^2 + \|x_{\epsilon}\|^2-\|x_{i_m}\|^2
\leq\|y_{\ell}\|^2+\Delta_{i_m}\bm{\epsilon}^2\|x_{\epsilon}\|^2.
$$ 
Consequently, utilizing $y_m\perp X_{m-1}$ from \eqref{defy}, 
and relations \eqref{xim}--\eqref{tildeg}, we obtain 
\begin{align*} 
	\cos\angle(\mathcal{X}_{m-1},x_{i_m})
	&=\frac{\|X_{m-1}^Tx_{i_m}\|}{\|x_{i_m}\|} 
	\leq \frac{\|X_{m-1}^TW_{\ell}\|\|\tilde g_{\ell}\|
		+\|X_{m-1}^TW_{\epsilon}\|
		\|\tilde g_{\epsilon}\| }{\|x_{i_m}\|} \nonumber \\
	&\leq \frac{\|X_{m-1}^TW_{\ell}\|(\|y_{\ell}\|
		+\sqrt{\Delta_{i_m}}\bm{\epsilon}\|x_{\epsilon}\|) 
		+ \Delta_{i_m}\bm{\epsilon}^2\|x_{\epsilon}\| }
	{(1-\Delta_{i_m}\bm{\epsilon}^2)\|x_{\epsilon}\|} \nonumber  \\ 
	&\leq 
	\frac{\|X_{m-1}^TW_{\ell}\|\|y_{\ell}\| }
	{(1- \Delta_{i_m}\bm{\epsilon}^2)\|x_{\epsilon}\|} 
	+ \frac{2\sqrt{\Delta_{i_m}}\bm{\epsilon}}
	{ 1-\Delta_{i_m}\bm{\epsilon}^2}, 
\end{align*}
where we have used $\|X_{m-1}^TW_{\ell}\|\leq1$,  
$\|X_{m-1}^TW_{\ell}\|\leq1$, and 
$\eta+\eta^2<2\eta$ for $\eta=\sqrt{\Delta_{i_m}}\bm{\epsilon}<1$.  
Finally, since $\tilde x_m=x_{i_m}/\|x_{i_m}\|$, 
substituting \eqref{g1g2} and 
$\|X_{m-1}^TW_{\ell}\|= \|\sin\angle(\mathcal{N}_{\epsilon}, 
\mathcal{X}_{m-1})\|$ into this relation establishes \eqref{cos2}. 
\end{proof}
 
Next, we establish the following accuracy relationship 
between $\mathcal{X}_m$ and $\mathcal{X}_{m-1}$.
 
\begin{lemma}\label{lemma2}
	With notations of Lemma~\ref{lemma1}, for $m=2,\cdots,n$, 
	it holds that 
	\begin{equation*}  
		\|\sin\angle(\mathcal{N}_{\epsilon},\mathcal{X}_{m})\|_F^2
		=  \|\sin\angle(\mathcal{N}_{\epsilon},\mathcal{X}_{m-1})\|_F^2
			+ \sin^2\angle(\mathcal{N}_{\epsilon},x_{m}^{\mathrm{c}}) .
	\end{equation*} 
\end{lemma}

\begin{proof}
	Since the columns of $X_{m-1}$, $X_{m}=[X_{m-1},x_{m}^{\mathrm{c}}]$   
	and $W_\ell$ form orthonormal bases for $\mathcal{X}_{m-1}$, 
	$\mathcal{X}_{m}$ and $\mathcal{R}_{\ell}$, respectively, 
	by definition, we have
	\begin{align*}
		\|\sin\angle(\mathcal{N}_{\epsilon},\mathcal{X}_{m})\|_F^2
		= &\| W_\ell^TX_{m}\|_F^2 = \|W_\ell^TX_{m-1}\|_F^2 + 
			\|W_\ell^Tx_{m}^{\mathrm{c}}\|^2\nonumber \\ 
		= &\|\sin\angle(\mathcal{N}_{\epsilon},\mathcal{X}_{m-1})\|^2
			+ \sin^2\angle(\mathcal{N}_{\epsilon},x_{m}^{\mathrm{c}}).
	\end{align*} 
\end{proof}

With the help of Lemmas~\ref{lemma1} and \ref{lemma2}, we can now 
establish the following result. 
 
\begin{theorem}\label{thm5}
Define $\vartheta_m = \max\limits_{1\leq j\leq m}
\angle(\mathcal{N}_{\epsilon},y_j)$ 
	and $\delta_m=\max\limits_{1\leq j\leq m}\Delta_{i_j}$ for 
	$1\leq m\leq n$. 
Suppose that $\vartheta_n<\frac{\pi}{2}$, 
$\delta_n\bm{\epsilon}^2\leq 1$, and 
the stopping tolerance $\varepsilon$ in \eqref{convergence} satisfies
\begin{equation}\label{condition}
 \varepsilon <\frac{1}{\kappa\tan\vartheta_n}
 \left(\frac{1-\delta_{n}\bm{\epsilon}^2}{n-1} 
 -2\sqrt{ \delta_{n}} \bm{\epsilon} \right),    
\end{equation} 
where $\kappa$ is the numerical condition number of $A$.  
Then, for $m=1,\dots n$, it holds that  
\begin{equation}\label{accXm}
\|\sin\angle(\mathcal{N}_{\epsilon},\mathcal{X}_{m})\|_F 
\leq \frac{\sqrt{m}\kappa\varepsilon}
{\sqrt{1-(m-1)\left(\frac{\tan\vartheta_m}{1-\delta_{m}\bm{\epsilon}^2} 
		\cdot \kappa\varepsilon +\frac{2\sqrt{ \delta_{m}}   
			\bm{\epsilon}}{1-\delta_{m}\bm{\epsilon}^2} \right)   }}.
\end{equation} 
\end{theorem}
 
\begin{proof} 
	Denote $\Theta_m = \|\sin\angle(\mathcal{N}_{\epsilon}, 
	\mathcal{X}_{m})\|_F$, $m=1,\dots,n$. 	
	Then \eqref{accXm} simplifies to   
	\begin{equation}\label{simplify}
		\Theta_{m}^2 \leq \frac{m\kappa^2\varepsilon^2}{1-(m-1)\omega_m}, 
	\end{equation}
where $\omega_m=\mu_m\kappa\varepsilon +\nu_m$ with 
$\mu_m=\frac{\tan\vartheta_m}{1-\delta_{m}\bm{\epsilon}^2}$ and 
$\nu_m=\frac{2\sqrt{\delta_{m}} \bm{\epsilon}}{1-\delta_{m}\bm{\epsilon}^2}$. 
Note that the monotonic increase of $\vartheta_m$ and $\delta_m$ 
implies the same for $\omega_m$. Consequently, condition 
\eqref{condition} guarantees that the denominator satisfies 
$1-(m-1)\omega_m \geq 1-(n-1)\omega_n > 0$.
 
We establish \eqref{simplify} by induction. For $m=1$, \eqref{simplify} 
reduces to $\Theta_{1} \leq\kappa\varepsilon$, which follows immediately 
from Theorem~\ref{thm2} since $X_1=x_{1}^{\mathrm{c}}=\tilde x_{1}$. 
Assume that \eqref{simplify} holds for $m-1\geq1$.  
For $m>1$, since $\|\sin\angle(\mathcal{N}_{\epsilon},\mathcal{X}_{m-1})\|
\leq \|\sin\angle(\mathcal{N}_{\epsilon},\mathcal{X}_{m-1})\|_F=\Theta_{m-1}$,  
\eqref{cos2} simplifies to 
$$
\cos\angle(\mathcal{X}_{m-1},\tilde x_{m})
\leq \Theta_{m-1}\cdot  \frac{\tan\vartheta_m}{1-\delta_{m}\bm{\epsilon}^2} 
+\frac{2\sqrt{  \delta_{m}}   \bm{\epsilon}}{1-\delta_{m}\bm{\epsilon}^2} =  
\mu_m\Theta_{m-1}+\nu_m. 
$$
Applying this to \eqref{accxm} and using 
$\sin\angle(\mathcal{N}_{\epsilon},\tilde x_{m})\leq\kappa\varepsilon$ 
from Theorem~\ref{thm2}, we obtain
\begin{equation*}  
	\sin\angle(\mathcal{N}_{\epsilon},x_{m}^{\mathrm{c}}) 
	\leq \frac{\kappa\varepsilon
		+ \Theta_{m-1}(\mu_m\Theta_{m-1}+\nu_m) }
	{\sqrt{1-(\mu_m\Theta_{m-1}+\nu_m)^2}}.  
\end{equation*}
Substituting this relation into Lemma~\ref{lemma2} yields   
\begin{align*}
	\Theta_{m}^2 &= \Theta_{m-1}^2
	+ \sin^2\angle(\mathcal{N}_{\epsilon},x_{m}^{\mathrm{c}})  
	\leq\frac{(1+2\mu_m\kappa\varepsilon)\Theta_{m-1}^2 
		+ \kappa^2\varepsilon^2 + 2\nu_m\kappa\varepsilon \Theta_{m-1} }
	{1-(\mu_m\Theta_{m-1}+\nu_m)^2}.    
\end{align*}
Exploiting the inductive hypothesis for $m-1$ and the property 
$\omega_{m-1} \leq \omega_m$, we deduce   
{\footnotesize \begin{align*}
	\Theta_{m}^2 &\leq   \frac{(1+2\mu_m\kappa\varepsilon) 
		\cdot \frac{(m-1)\kappa^2\varepsilon^2}{1-(m-2)\omega_{m}} 
		+ \kappa^2\varepsilon^2 + 2\nu_m\kappa\varepsilon 
		\cdot\tfrac{\sqrt{m-1}\kappa\varepsilon}{\sqrt{1-(m-2)\omega_{m}}}}
	{1-(\mu_m\cdot \frac{\sqrt{m-1}\kappa\varepsilon}
		{\sqrt{1-(m-2)\omega_{m}}}+\nu_m )^2} \nonumber \\ 
	&=\kappa^2\varepsilon^2 \cdot \frac{m+m\omega+2\nu_m\sqrt{m-1}
		\left(\sqrt{1-(m-2)\omega_m}-\sqrt{m-1}\right)}   
	{1-(m-2)\omega_{m}-\left(\omega_m \sqrt{m-1} + 
		\nu_m(\sqrt{1-(m-2)\omega_{m}}-\sqrt{m-1})\right)^2} \nonumber \\  
	&\leq \kappa^2\varepsilon^2 \cdot \frac{m+m\omega_m}   
	{1-(m-2)\omega_{m}-\left(\omega_m\sqrt{m-1}\right)^2} =
	\kappa^2\varepsilon^2 \cdot \frac{m}{1-(m-1)\omega_m },  
\end{align*}}
\hspace{-0.5em}
where the equality follows by multiplying the numerator and 
denominator by $1-(m-2)\omega_m$ and substituting 
$\mu_m\kappa\varepsilon=\omega_m-\nu_m$. 
The subsequent inequality holds since 
$\sqrt{1-(m-2)\omega_{m}}<\sqrt{m-1}$. 
Thus, \eqref{simplify} holds for $m$, completing the induction. 
\end{proof}

Theorem~\ref{thm5} extends the accuracy bound established in 
Theorem~\ref{thm2} for a single approximate numerical null 
vector to an $m$-dimensional approximate numerical null subspace 
for $m=1,\dots,n$. 
In particular, \eqref{accXm} reduces to \eqref{accxke} when $m=1$. 
Given the random vectors $\hat y_1,\dots,\hat y_n$ and the sequence 
$y_1,\dots,y_n$ defined in \eqref{defy}, the estimate \eqref{xm} 
suggests that $\tan\vartheta_n\approx\sqrt{N-n}$. 
Moreover, \eqref{delta} indicates that the parameter 
$\delta_m\lesssim\mathcal{O}(\kappa^2)$. 
Consequently, condition \eqref{condition} is readily fulfilled 
provided that the stopping threshold $\varepsilon$ in 
\eqref{convergence} is sufficiently small to satisfy \eqref{sig2}.  
In this case, \eqref{accXm} simplifies to  
\begin{equation}\label{accXm2}
	\|\sin\angle(\mathcal{N}_{\epsilon},\mathcal{X}_{n})\|_F \leq
	\sqrt{n}\kappa\varepsilon(1+\mathcal{O}(\kappa\varepsilon)).
\end{equation} 
 
\subsection{Termination}\label{subsec:3-2}
In practice, the dimension $n$ of the numerical null space  
$\mathcal{N}_{\epsilon}$ is generally unknown \textit{a priori}.  
It must be determined adaptively during the computation 
to ensure proper termination of the LSQRNS algorithm. 

Assume that $n$ approximate, orthonormal numerical null vectors 
$x_{1}^{\mathrm{c}}$, $\dots$, $x_{n}^{\mathrm{c}}$ of $A$ have 
already been computed. For the $(n+1)$-th starting vector 
$y_{n+1}=\frac{\bar y_{n+1}}{\|\bar y_{n+1}\|}$ defined in \eqref{defy}, 
where $\bar y_{n+1}=(I-X_nX_n^T)\hat y_{n+1}$, we have 
\begin{align*} 
	\sin\angle(\mathcal{R}_{\ell},y_{n+1}) 
	& =\sin\angle(\mathcal{R}_{\ell},\bar y_{n+1}) 
	= \frac{\|W_{\epsilon}^T(I-X_nX_n^T)\bar y_{n+1}\|}
	  {\|\bar y_{n+1}\|} \nonumber \\
	& \leq \| W_{\epsilon}^T(I-X_nX_n^T)\| 
	= \|\sin\angle(\mathcal{N}_{\epsilon},\mathcal{X}_{n})\|
	=\mathcal{O}(\varepsilon).
\end{align*} 
Invoking Remark~\ref{remark3-4}, we deduce that during the $(n+1)$-th 
outer cycle of LSQRNS, the monotonically decreasing norm $\|x_i\|$ of 
the LSQRNV iterate will drop below $\mathcal{O}(\varepsilon)$ well 
before the convergence criterion \eqref{convergence} is met.
Conversely, during the $m$-th outer cycle for $m \leq n$, the bound 
\eqref{orthdecompose} and the estimate \eqref{xm} guarantee that this 
norm remains bounded from below by $(1-\mathcal{O}(\bm{\epsilon}^2))\|x_m^*\| 
=  (1-\mathcal{O}(\bm{\epsilon}^2)) \cos\angle(\NN_{\epsilon},y_{m})  
\approx\frac{\sqrt{n-m+1}}{\sqrt{N-m+1}}\geq\frac{1}{\sqrt{N-n+1}}$.
Thus, a significant drop in $\|x_i\|$ below $\frac{1}{\sqrt{N-m+2}}$ 
reliably indicates that $m=n+1$.   
Consequently, by virtue of \eqref{resLSQR2}, if the quantity
$|\varphi_i|=\|x_i\|$ generated during 
the $m$-th outer cycle of LSQRNS satisfies 
\begin{equation}\label{terminate}
	|\varphi_{i}|\leq0.001/\sqrt{N-m+2}, 
\end{equation}
we conclude that all $n=m-1$ orthonormal basis vectors have been found. 
We then terminate LSQRNS and return the basis matrix  
$X_{n}=[x_{1}^{\mathrm{c}}, \cdots,x_{n}^{\mathrm{c}}]$.

Notably, the final cycle of LSQRNS is typically more 
economical than the preceding $n$ ones, as it halts via 
\eqref{terminate} long before satisfying \eqref{convergence}. 
Furthermore, by Remark~\ref{remark2-1}, if LSQRNV breaks 
down with $\beta_{i}=0$ at any stage $i$ during cycle $m$, 
it simply produces $\tilde x_{m}=\bm{0}$, triggering immediate 
termination of LSQRNS.  

Algorithm~\ref{alg3} details three modifications to Algorithm~\ref{alg2} 
for the $m$-th outer cycle of LSQRNS. 
A return value of $\mathrm{flag}_m=0$ indicates normal convergence to a 
valid approximate numerical null vector $\tilde x_{m}$. Alternatively, 
$\mathrm{flag}_m=-1$ signals termination due to \eqref{terminate} or 
a breakdown, thereby identifying the dimension $n=m-1$. 
 
\begin{algorithm}[htbp]
	\caption{Modifications to Algorithm~\ref{alg2} in the $m$th 
		outer cycle of LSQRNS}\label{alg3}
	\begin{itemize}[leftmargin=0.6cm]
		\item[\footnotesize \ref{step1}:]  \textbf{if} $\beta_i=0$ 
		\textbf{then} set $\mathrm{flag}=-1$ and break;
		\textbf{else} compute $v_{i+1}=\frac{r_i}{\beta_i}$. \textbf{fi}
		
		\item[\footnotesize  \ref{step2}:] \textbf{if}  
		$\varphi_i \!\leq\!\frac{0.001}{\sqrt{N-m+2}}$ \textbf{then} 
		set $\mathrm{flag}\!=\!-1$ and break;
		\textbf{else} update $x_i=s_i x_{i-1}\!+\!c_iv_{i+1}$ \textbf{fi}
		
		\item[\footnotesize  \ref{step3}:]  \textbf{if} $|\alpha_{i+1}c_{i}
		|\leq  \tilde\sigma_1\varepsilon$  
		\textbf{then} set $\mathrm{flag}=0$ and break; \textbf{else} 
		compute $u_{i+1}=\frac{t_i}{\alpha_{i+1}}$. \textbf{fi} 
	\end{itemize} 
\end{algorithm} 
  
\subsection{The LSQRNS algorithm: a pseudocode}\label{subsec:3-3} 

\begin{algorithm}[htbp]
	\caption{LSQRNS for the null space problem  of $A$}
	\begin{algorithmic}[1]\label{alg4}
		\STATE{Initialization: Chose $\gamma>0$ and set $X_{0}=[\ ]$.}
		
		\FOR{$m=1,\cdots,m_{\max}\leq N$}
		
		\STATE{\label{step5-3} Generate the random vector 
			$\hat y_m\in\mathbb{R}^N$ with  
			$\EE(\hat y_m)=\bm{0}$ and 
			$\cov(\hat y_m)=\gamma I$, and compute 
			$y_m=\left(I-X_{m-1}X_{m-1}^T\right)\hat y_m/
			\|\left(I-X_{m-1}X_{m-1}^T\right)\hat y_m\|$.}  
		
		\STATE{\label{step5-4}Implement Algorithm~\ref{alg3} 
			with the initial vector $y_m$,
			and obtain $\mathrm{flag}_m$ and $\tilde x_{m}$.} 
		
		\STATE{\textbf{if} $\textit{flag}_m = -1$ 
			\textbf{then} return $n=m-1$ and $X=X_{n}$. \textbf{fi}}
		
		\STATE{\label{step5-6}Compute $ x_{m}^{\mathrm{c}}=\frac{(I-X_{m-1}
				X_{m-1}^T)\tilde x_{m}}{\|(I-X_{m-1}X_{m-1}^T)\tilde x_{m}\|}$ 
			and update $X_{m}=[X_{m-1},x_{m}^{\mathrm{c}}]$.}
		
		\ENDFOR 
		
		\STATE{Return $n=m_{\max}$ and $X=X_n$.}	 
	\end{algorithmic}
\end{algorithm} 

Algorithm~\ref{alg4} outlines the proposed LSQRNS algorithm with deflation. 
During the first cycle with $m=1$, the bidiagonal projected matrix $B_{30}$ 
defined in \eqref{LBDmat} is stored. The full SVD of $B_{30}$ is then computed, 
and its largest singular value $\tilde\sigma_1$ serves as a valid approximation 
to $\sigma_1$, the largest singular value of $A$, and is incorporated into 
the inner stopping criterion \eqref{conLSQR3} for all subsequent cycles; 
see step~8 of Algorithm~\ref{alg3}.
 
By setting the maximum number of outer cycles, $m_{\max}$, to 
the default value $N$, the LSQRNS algorithm adaptively determines 
the dimension $n$ of the numerical null space $\mathcal{N}_{\epsilon}$ 
for a large matrix $A$ with respect to 
$\bm{\epsilon}\leq \mathcal{O}(\varepsilon)$, assuming the corresponding 
numerical condition number satisfies $\kappa\ll\varepsilon^{-1}$ 
(cf. \eqref{sig2}).   
Simultaneously, it computes an approximate orthonormal basis matrix $X_n$
for $\mathcal{N}_{\epsilon}$, ensuring that the corresponding subspace 
$\mathcal{X}_n=\spans\{X_n\}$  adheres to \eqref{accXm2}. 
Particularly, if $n$ is known \textit{a priori} in certain applications, 
one may set $m_{\max}=n$ to avoid the redundant final cycle and 
terminate the algorithm early. 
Furthermore, by setting $m_{\max}=1$, LSQRNS serves as an efficient 
detector for numerical rank deficiency:  an empty output $X$ 
indicates full column rank, whereas a non-empty $X$ confirms numerical 
rank deficiency of $A$.

\section{Numerical experiments}\label{sec:5}
We report numerical experiments on several problems to illustrate the 
performance of the LSQRNS algorithm proposed in this paper, which was 
coded in MATLAB. All the experiments were performed on a 13th 
Gen Intel(R) Core(TM) i9-13900KF 3.00 GHz workstation with 64GB RAM 
using the MATLAB R2022b platform with the machine precision 
$\bm{\epsilon}=2.22\times10^{-16}$ 
under the Miscrosoft Windows 11 Professional system. 
 
\subsection{Performance in Null Space Computation}

\begin{table}[tbhp] 
	\caption{Properties of the test matrices: part I.}\label{table0}
	\begin{center}
		\begin{tabular}{ccccccc} \toprule
			{$A$}&$M$&$N$&$\mathrm{nnz}$& $n$ & $\sigma_{1}$&$\kappa$ \\ \midrule 
			shar\_te2-b2 &200200&17160&600600&285&8.06& 1.59  \\ 
			relat8 &345688&12347&1334038&58&18.8& 15.7  \\ 
			se &32768&32768&98300&457&3.00& 255  \\ 
			stormg2-27* &37485 &14441 &94274 &54 &545  &3.99e+3 \\
			k1\_san &67759 &67759 &559774 &1 &136 &6.52e+5 \\
			barth &6691&6691&26439&1&4.89& 8.28e+5  \\   
			M80PI\_n1 &4028&4028&9927&3&3.98&2.34e+6\\  
			pdb1HYS &36417&36417&4344765&6&352& 2.55e+6  \\ 
			Franz11 &47104  &30144 &329728  &9151 &7.48 &6.85 \\
			GL7d14  &171375 &47271 &1831183 &7939 &17.5 &10.3 \\
			\bottomrule
		\end{tabular}
	\end{center}
\end{table}
 
To evaluate the practical effectiveness of LSQRNS, we selected a diverse 
set of test matrices from the SuiteSparse Matrix Collection 
\cite{davis2011university}, representing various realistic applications. 
Specifically, the test cases include:  
(\romannumeral1) computational algebra and topology 
(e.g., ``shar\_te2-b2'', ``relat8'',  ``Franz11'', ``GL7d14''), 
where null spaces are essential 
for computing cohomology groups and uncovering the topological invariants 
of abstract algebraic structures;
(\romannumeral2) optimization and control systems 
(e.g., ``stormg2-27*'', ``M80PI\_n1''), where null-space 
computations are critical for determining feasible search 
directions and identifying invariant subspaces for model reduction
(\romannumeral3) structural and fluid dynamics 
(e.g., ``k1\_san'', ``barth''), where null spaces 
are instrumental in detecting structural instabilities and resolving constrained fluid flows;  
(\romannumeral4) complex network analysis (e.g., , ``se''), 
where null vectors facilitate the partitioning of 
topological systems;  
and (\romannumeral5) statistical and molecular modeling 
(e.g., ``pdb1HYS''), where null-space analysis reveals 
linear redundancies in datasets as well as the conformational 
flexibility of macromolecules. 
By design, these test cases feature the severe ill-conditioning, 
high dimensionality, and structural complexity characteristic 
of large-scale engineering problems. 
The fundamental properties of the matrices are summarized in Table~\ref{table0}, 
where $\mathrm{nnz}$ denotes the number of nonzero entries, while 
$n$ and $\kappa$ represent the numerical nullity and condition number, 
respectively. 
For matrices of moderate size, $\sigma_1$, $\kappa$ and $n$ were computed 
using the MATLAB built-in function \texttt{svd} with a truncation 
tolerance of $\sigma_1\cdot 10^{-14}$. 
For the exceptionally large instance, ``k1\_san'', these quantities 
were derived by computing the largest singular value via the MATLAB 
built-in function \texttt{svds}, and several of the smallest singular 
values using the JDSVD-V\_HYBRID algorithm \cite{huang2024preconditioning} 
(MATLAB code available at \url{https://github.com/JinzhiHuang/JDSVD-V}). 
  
\begin{exper}\label{exper1}
	We computed the orthonormal bases for the numerical null spaces of 
	the first eight test matrices in Table~\ref{table0}. 
	For this purpose, we employed our LSQRNS algorithm 
	alongside the randomized small-block Lanczos (TRLanczos) algorithm 
	proposed in \cite{kressner2026randomized} (MATLAB code available at 
	\url{https://github.com/nShao678/Nullspace-code}). 
	In our numerical tests, the stopping tolerance for LSQRNS was set to 
	$\varepsilon = 10^{-10}$, while the rounding and stopping threshold for 
	TRLanczos was configured as $\mathrm{eps}0 = \min\{\sigma_1\cdot 10^{-10}, 
	\sigma_{\ell}^2\cdot 10^{-2}\}$.      
\end{exper}
 
For LSQRNS, we set $m_{\max}=N$ and generated the random seed vectors 
$\hat y_m$ ($m \ge 1$) via MATLAB's \texttt{randn}. 
For TRLanczos, to guarantee convergence and minimize the 
total iterations and execution time, we set 
the maximum Krylov basis size to $\dim \mathcal{K}=1024$. 
Following \cite{kressner2026randomized}, the corresponding 
block size was set to $d=\dim \mathcal{K}/16$.  
Table~\ref{table1} summarizes the computational results, where 
``MVs'' denotes the total matrix-vector products with 
$A$ or $A^T$, and ``Time'' records the CPU time in seconds measured 
via MATLAB's \texttt{tic} and \texttt{toc}. 
The metric $\mathrm{RelRes} = \|AX\|/\|A\|$ represents the relative 
residual norm of the computed approximate orthonormal basis matrix $X$. 
Analogous to Theorem~\ref{thm2}, the inequality 
$\sin\angle(\mathcal{N}_{\epsilon}, \mathcal{X})
\leq\kappa\cdot\mathrm{RelRes}$ holds, 
implying that $\mathrm{RelRes}$ effectively quantifies the numerical 
accuracy of $\mathcal{X}$. 
Finally, a hyphen "$-$" indicates that the algorithm failed to converge 
within an $8$-hour limit and was manually terminated, as TRLanczos 
lacks a built-in stopping criteria for MVs or time.   
   
	   
\begin{table}[tbhp]  \small
	\caption{Computational results of Experiment~\ref{exper1}.}\label{table1} 
	\begin{center} 
		\begin{tabular}{ccccccccc} \toprule
			\multirow{2}{*}{$A$} &\multicolumn{4}{c}{LSQRNS}
			&\multicolumn{4}{c}{TRlanczos} 
			\\ \cmidrule(lr){2-5}\cmidrule(lr){6-9} 
			&n &MVs&Time& RelRes  &n &MVs &Time& RelRes    \\
			\midrule
			shar\_te2-b2  &285  &2879     &5.25    &2e-16 &285 &4736  &7.10  &6e-12  \\
			relat8       &58   &8249     &17.0    &2e-10 &58  &6144  &9.28  &1e-11  \\
			se           &457  &1064366  &132     &2e-10 &64  &91648  &172  &2e-8   \\
			stormg2-27*  &54   &196202   &21.7    &2e-10 &54  &41472  &66.3  &3e-11   \\
			k1\_san      &1    &116578   &55.2    &1e-10 &-  &-  &-  &-    \\
			barth        &1    &55495    &3.91    &1e-10 &1  &44544  &18.4  &3e-8   \\  
			M80PI\_n1    &3    &89560    &1.49    &1e-10 &3  &10240  &2.98  &2e-10   \\ 
			pdb1HYS      &6    &915979   &3.22e+3 &1e-10 &6 &230400   &456  &3e-9   \\ 
			\bottomrule
		\end{tabular}  
	\end{center}
\end{table}
  
As shown, LSQRNS correctly determined the numerical null-space dimension 
$n$ for all eight test matrices and computed highly accurate 
orthonormal bases, consistently achieving relative residual norms 
between $10^{-16}$ and $10^{-10}$. In contrast, TRLanczos exhibited 
notable reliability and accuracy issues. 
For instance, it returned an incorrect dimension for the matrix ``se'', 
yielding $n=64$, which merely reflects the initial block size rather 
than the true dimension of $457$---a failure that persists even when  
$\dim\mathcal{K}$ is increased to $2048$.  
Furthermore, it failed to converge within the $8$-hour limit for 
``k1\_san'', and generally yielded less accurate bases, producing 
relative residual norms of $2 \times 10^{-8}$ and $3 \times 10^{-8}$ 
for ``se'' and ``barth'', respectively. 
Additionally, the results reveal a striking discrepancy between the MV 
count and the overall CPU time for TRLanczos. 
For matrices such as ``stormg2-27*'', ``barth'', and ``M80PI\_n1'', 
TRLanczos required fewer MVs than LSQRNS yet consumed significantly more 
execution time. 
This inefficiency stems directly from the necessity of maintaining an 
impractically large Krylov basis with $\dim\mathcal{K}=1024$, which 
incurs prohibitive dense reorthogonalization costs and demands an 
auxiliary storage of $\mathcal{O}(1024N)$ in addition to storing $A$. 
In fact, adopting a more conventional dimension of $\dim\mathcal{K}=128$ 
caused severe performance deterioration across all eight test matrices. 
This reduction compromised not only computational efficiency but also 
numerical reliability, ultimately yielding incorrect null-space dimensions.    
Conversely, LSQRNS relies on short recurrences and updates the 
approximations on the fly. This strictly limits its additional 
storage requirement to $\mathcal{O}(2N+M)$ and keeps the per-iteration 
overhead negligible compared to the cost of the matrix-vector products. 
Although TRLanczos converged faster on specific test cases such as 
``relat8'' and ``pdb1HYS'', its poor robustness and excessive memory 
demands severely limit its practical utility. 
Ultimately, LSQRNS provides a fundamentally more reliable, 
memory-efficient, and structurally robust framework for large-scale 
numerical null-space computation. 
 
Moreover, as shown in Table~\ref{table1}, LSQRNS exhibited rapid 
convergence for the first four well-conditioned matrices, requiring 
only a few MVs to compute each basis vector for the numerical null space. 
In contrast, for the remaining four matrices, which are moderately to 
severely ill-conditioned as detailed in Table~\ref{table0}, the algorithm 
consumed significantly more MVs per outer cycle. 
These corroborate  the convergence results established in 
Theorem~\ref{thm1} and  the subsequent analysis. 
 
\begin{figure}[thbp]
  \centering 
  \includegraphics[width=0.85\textwidth]{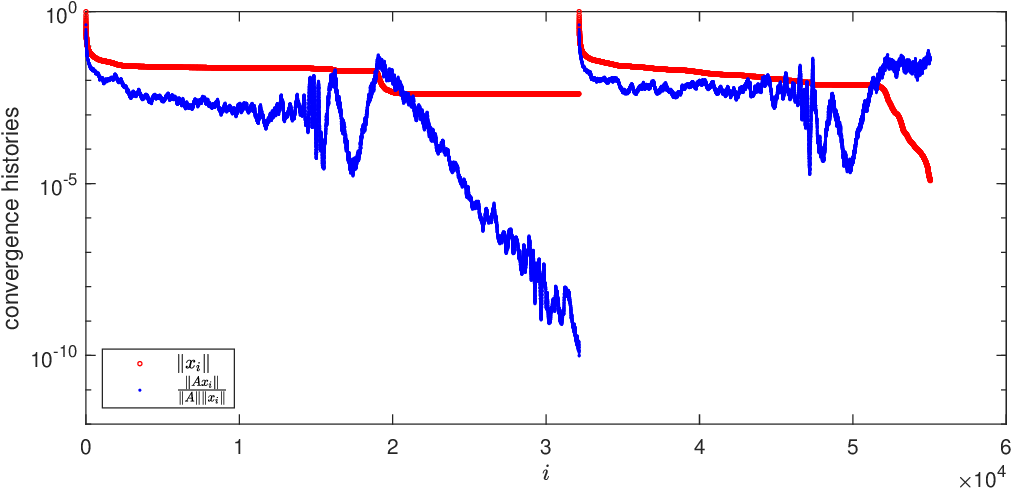}
  \caption{Convergence curve of LSQRNS for ``barth''.}\label{fig1}
\end{figure}

Particularly, for the matrix ``barth'', Figure~\ref{fig1} depicts the 
convergence curves of the LSQRNS iterates $x_i$ in terms of both the 
approximate solution norm $\|x_i\|$ and the relative residual norm 
$\|Ax_i\|/(\|A\|\|x_i\|)$. 
Given the nullity $n=1$, the convergence process exhibited a distinct 
two-phase pattern. 
As expected,  $\|x_i\|$ decreased monotonically within each cycle. 
In the first cycle, the relative residual norm fell below the stopping 
tolerance $\varepsilon$ at iteration $32156$, whereas $\|x_i\|$ 
stagnated near $4.07\times 10^{-3}$. In the second cycle, $\|x_i\|$ 
reached the termination threshold of $1.22\times10^{-5}$ at 
iteration $22894$, 
with the relative residual norm oscillating above $10^{-2}$. 
This validates the termination criterion \eqref{terminate}. 
Furthermore, the second cycle required $28.80\%$ fewer MVs than 
the first one, corroborating the efficiency gains discussed in 
Section~\ref{subsec:3-2}.  
  
\begin{exper}\label{exper2}
	We evaluated the LSQRNS and TRLanczos algorithms on the matrices 
	``Franz11'' and ``GL7d14'', both of which feature large nullities 
	as shown in Table~\ref{table0}. 
	The stopping tolerance for LSQRNS was set to $\varepsilon=10^{-8}$, 
	while the rounding and stopping threshold for TRLanczos was 
	configured as $\mathrm{eps}0= \sigma_1\cdot 10^{-8}$. 
\end{exper}

\begin{table}[tbhp]  \small 
	\caption{Computational results of Experiment~\ref{exper2}.}\label{table4} 
	\begin{center} 
		\begin{tabular}{ccccccccc} \toprule
			\multirow{2}{*}{Algorithm} &\multicolumn{4}{c}{Franz11}
			&\multicolumn{4}{c}{GL7d14} 
			\\ \cmidrule(lr){2-5}\cmidrule(lr){6-9} 
			&n &MVs&Time&\!RelRes\!  &n &MVs &Time&\!RelRes\!    \\
			\midrule
			LSQRNS &9151 &374867 &4.43e+3 &2e-7 &7939 &562105 &6.50e+3 &2e-7  \\
			TRL(4096) &3328 &100608 &573 &2e-9 &3328 &120576 &853 &1e-9  \\
			TRL(8192) &6656 &172544 &1.97e+3 &2e-9 &6656 &204288 &3.16e+3 &1e-9  \\
			\!TRL(10240)\! &8320 &203520 &3.96e+3 &2e-9 &7939 &225920 &4.31e+3  &1e-9  \\
			\!TRL(12288)\! &9151 &208128 &3.95e+3 &2e-9 &7939 &211968 &6.00e+3 &1e-9  \\ 		  
			\bottomrule
		\end{tabular}  
	\end{center}
\end{table}

Table~\ref{table4} summarizes the computational results of LSQRNS 
and TRLanczos (denoted by ``TRL'', with user-defined subspace 
dimensions in parentheses). As observed, LSQRNS successfully 
captured the complete numerical null spaces for both matrices, 
identifying exactly $n=9151$ and $7939$ numerical null vectors, 
respectively. 
These results reveal a fundamental dilemma for projection-based 
methods when dealing with large nullities. For instance, to 
achieve complete recovery, TRLanczos required a prohibitively 
large subspace dimension  (e.g., $12288$ for ``Franz11'' and 
$10240$ for ``GL7d14''). 
However, when the subspace was restricted to mitigate memory 
exhaustion with $\dim\mathcal{K}=4096$, TRLanczos terminated 
prematurely, capturing  $3328$ numerical null vectors. 
In realistic large-scale computations, maintaining and 
reorthogonalizing even a $4096$-dimensional basis poses a 
formidable memory and computational burden, making any 
further expansion practically impossible. 
LSQRNS effectively circumvents this critical bottleneck. 
Although TRLanczos can exhibit slightly shorter execution times 
when artificially provided with an optimal search subspace, 
LSQRNS requires no \textit{a priori} guess of the true nullity. 
It extracts the complete numerical null space with a strictly 
bounded memory footprint, rendering it fundamentally more 
reliable for high-nullity problems. 
 
\subsection{Robustness in rank-deficiency detection} 

To assess the reliability of LSQRNS in detecting rank deficiency, we 
investigated eight large-scale matrices selected from the SuiteSparse 
Matrix Collection \cite{davis2011university}. Stemming from diverse 
applications, these test cases range from directed graphs and circuit 
simulations to linear programming and combinatorial problems. 
In practice, distinguishing full-rank systems from rank-deficient 
ones is crucial for the numerical stability of downstream solvers.
Table~\ref{table2} summarizes the fundamental properties of these 
matrices. As shown, the first four matrices are numerically full-rank, 
while the last four are severely rank-deficient. For accurate reference, 
the extreme singular values $\sigma_1$ and $\sigma_N$ were computed 
using MATLAB's \texttt{svds} and the JDSVD-V\_HYBRID algorithm 
\cite{huang2024preconditioning}, respectively.
 
\begin{table}[tbhp] 
 	\caption{Properties of the test matrices: part II.}\label{table2}
 	\begin{center}
 		\begin{tabular}{cccccc} \toprule
 			
 			$A$&$M$&$N$&nnz&$\sigma_{1}$&$\sigma_{N}$  \\ \midrule      
 			cage15       &5154859  &5154859  &99199551  &1.02    &8.56e-2    \\
 			LargeRegFile &2111154  &801374   &4944201  &3.06e+3  &2.75e-1    \\  
 			Rucci1       &1977885  &109900   &7791168  &7.07     &1.04e-3   \\
 			Hamrle3      &1447360  &1447360  &5514242  &16.1     &8.01e-4    \\ 
 			relat9       &12360060 &549336   &38955420  &21.6    &4.47e-20    \\ 
 			rel9         &9888048  &274669   &23667183  &21.1    &1.85e-17   \\
 			wheel\_601   &902103   &723605   &2170814  &25.1     &3.02e-19    \\ 
 			lp1          &534388   &534388   &1643420  &501      &1.30e-21    \\
 			\bottomrule
 		\end{tabular}
 	\end{center}
\end{table}

\begin{exper}\label{exper3}
We applied LSQRNS to the matrices listed in Table~\ref{table2}, 
using a stopping tolerance of $\varepsilon = 10^{-12}$, and a 
maximum of one outer cycle. The numerical rank was determined 
as described in section~\ref{subsec:3-3}. 
As a benchmark, we computed the smallest singular value $\sigma_N$ 
using JDSVD-V\_HYBRID with the same tolerance.  
A matrix is considered full-rank if $\sigma_N\gg\sigma_1\varepsilon$, 
and rank-deficient otherwise.  
\end{exper}
 
\begin{table}[tbhp]
	\caption{Computational results of Experiment~\ref{exper3}.}\label{table3}
	\begin{center}
		\begin{tabular}{ccccccc} \toprule
			\multirow{2}{*}{$A$}&\multicolumn{3}{c}{RLSQRNS}&\multicolumn{3}{c}{JDSVD-V}
			\\ \cmidrule(lr){2-4}\cmidrule(lr){5-7}
			& MVs & Time & Rank & MVs & Time & Rank \\	\midrule
			cage15        &56 &6.56 &F &318  &56.5  &F \\ 
			LargeRegFile  &971 &14.6 &F &3676  &64.5  &F \\
			Rucci1        &15584 &256 &F &101427 &1.73e+3 &F \\
			Hamrle3       &28914 &338 &F &200000 &5.13e+3 &F \\
			relat9        &185 &29.0 &D &904 &242 &D \\
			rel9          &249 &20.5 &D &1053 &167 &D \\ 
			wheel\_601    &5004 &33.1 &D &59684 &1.24e+3 &D \\
			lp1           &18538 &72.3 &D &119793 &1.55e+3 &D \\ 
			\bottomrule
		\end{tabular}
	\end{center} 
\end{table}

The computational results are summarized in Table~\ref{table3}, where 
``F'' and ``D'' in the ``Rank'' column denote numerically full-rank 
and rank-deficient matrices, respectively.  
As shown, LSQRNS accurately identified the rank properties of all test 
cases, aligning with the $\sigma_N$ baseline established by 
JDSVD-V\_HYBRID (see Table~\ref{table2}). 
Compared to JDSVD-V\_HYBRID, LSQRNS converged significantly faster, 
requiring substantially fewer MVs and less CPU time.  
Particularly, for the matrix ``Hamrle3'', JDSVD-V\_HYBRID failed to 
reach the desired accuracy within the maximum limit of 200,000 MVs, 
indicating a severe clustering of its smallest singular values. 
In contrast, LSQRNS converged rapidly for this matrix. 
These results demonstrate that for large-scale matrices, one should 
use specialized algorithms like LSQRNS to solve null-space problems, 
rather than relying on general-purpose eigenvalue or SVD solvers.
 
\begin{figure}[t]
	\centering
	\includegraphics[width=0.48\textwidth]{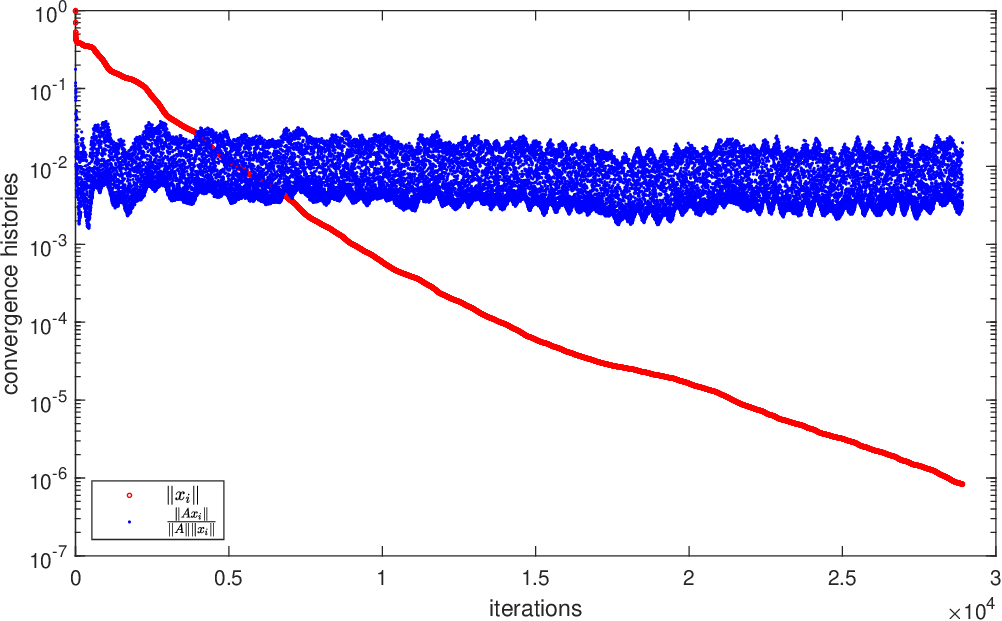}
	\includegraphics[width=0.48\textwidth]{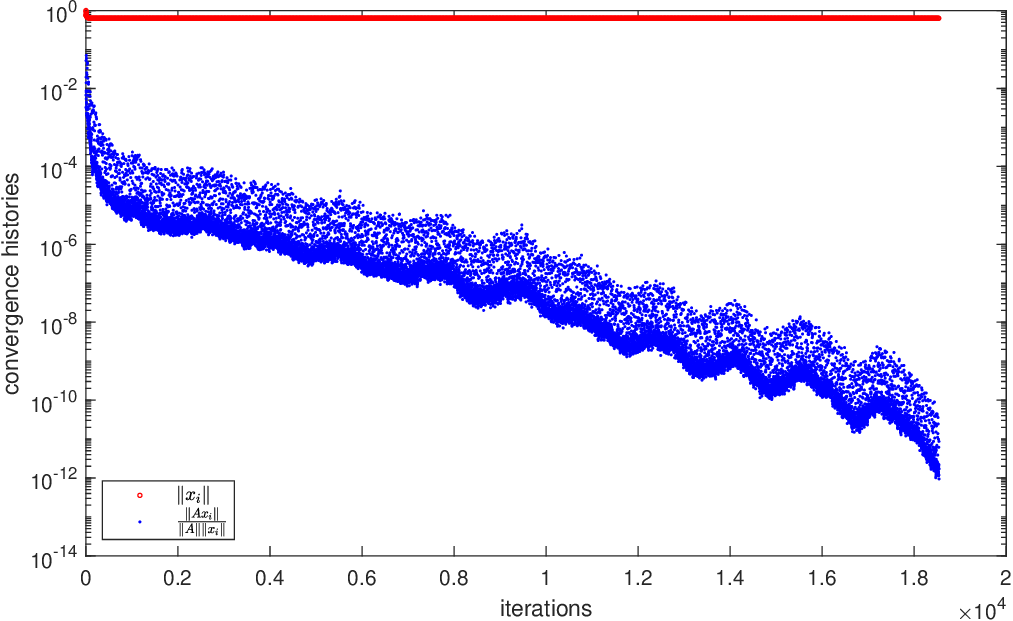}
	\caption{Convergence curves of LSQRNS for ``Hamrle3'' (left) 
		and ``lp1'' (right).}\label{fig2}
\end{figure}

Figure~\ref{fig2} depicts the convergence profiles of 
LSQRNS for two highly challenging cases: ``Hamrle3'' (left) and ``lp1'' (right). 
The erratic convergence behavior evident in Table~\ref{table3} and 
Figure~\ref{fig2} corroborates their severe ill-conditioning.  
Despite such numerical adversities, LSQRNS maintains its robustness. 
Specifically, for the full-rank matrix ``Hamrle3'', it 
successfully halted via the early termination criterion \eqref{terminate}. 
In contrast, for the rank-deficient matrix ``lp1'', it continued 
iterating until the convergence criterion \eqref{convergence} was met. 
These highlight the robustness of LSQRNS and the reliability of 
its termination strategies when tackling highly ill-conditioned problems. 
 
\section{Conclusions}\label{sec:6}
In this paper, we proposed the LSQRNV algorithm to compute a numerical 
null vector for a large-scale, rank-deficient matrix $A$ from an 
arbitrary initial vector. 
We provided a rigorous convergence analysis of the method and established 
explicit accuracy bounds for the computed null vectors. 
For practical applications, we incorporated robust deflation 
techniques and reliable stopping criteria to develop the LSQRNS algorithm. 
This framework effectively determines the dimension of the numerical null 
space of $A$ and constructs an orthonormal basis for this space, 
for which the corresponding accuracy bounds are also established.  
Furthermore, with appropriate parameter settings, LSQRNS serves 
as a highly efficient tool for detecting column rank deficiency. 
Extensive numerical experiments validate our theoretical results 
and demonstrate the efficiency and effectiveness of the proposed 
algorithms.
 
Furthermore, the proposed LSQRNS algorithm possesses inherent 
structural advantages that are highly amenable to parallel computing. 
A well-designed parallel implementation is expected to yield 
substantial gains in computational efficiency.  
This prospect is particularly appealing when the matrix in 
question is ill-conditioned or features a large-dimensional 
null space. This constitutes our future work.
 
\section*{Declarations}

The author declares that she has no financial interests, 
and she has read and approved the final manuscript.  
All numerical experiments can be reproduced using the MATLAB
implementations of the LSQRNS and LSQRNV algorithms, which are publicly
available at   \url{https://github.com/JinzhiHuang/LSQRNS}.

 
\bibliographystyle{siamplain}
\bibliography{nullspace_ref}
\end{document}